\documentclass[graybox]{svmult}

\usepackage{mathptmx}       
\usepackage{helvet}         
\usepackage{courier}        
\usepackage{type1cm}        
\usepackage{makeidx}         
\usepackage{graphicx}        
\usepackage{multicol}        
\usepackage[bottom]{footmisc}

\makeindex             

\usepackage{amsmath, amssymb, mathtools, hyperref}
\usepackage[shortlabels]{enumitem}
\usepackage{graphicx}
\usepackage{caption}
\usepackage{subcaption}
\usepackage{algorithmicx}
\usepackage{algpseudocode}
\usepackage{algorithm}

\newcommand{\R}{\mathbb{R}}

\newcommand{\C}{\mathbb{C}}

\newcommand{\PP}{\mathbb{P}}

\begin{document}

\title*{How to flatten a soccer ball}
\author{Kaie Kubjas, Pablo A. Parrilo and  Bernd Sturmfels}
\institute{Kaie Kubjas \at Dept.~of Mathematics and Systems Analysis, Aalto University, Finland, \email{kaie.kubjas@aalto.fi}
\and Pablo A. Parrilo \at Laboratory for Information and Decision Systems,
Massachusetts Institute of Technology,
Cambridge, USA, \email{parrilo@mit.edu}
\and Bernd Sturmfels \at Dept.~of Mathematics,
University of California, Berkeley, USA, \email{bernd@berkeley.edu}
}
%
%
\maketitle

\abstract*{This is an experimental case study in real algebraic geometry, aimed at
computing the image of  a semialgebraic subset of $3$-space
under a polynomial map into the plane.
For general instances, the boundary of the image  is given by two 
highly singular curves. We determine these curves and show how they demarcate
the ``flattened soccer ball''.
We explore cylindrical algebraic decompositions, by working through concrete examples.
Maps onto convex polygons and connections to convex optimization are also discussed.}

\abstract{This is an experimental case study in real algebraic geometry, aimed at
computing the image of  a semialgebraic subset of $3$-space
under a polynomial map into the plane.
For general instances, the boundary of the image  is given by two 
highly singular curves. We determine these curves and show how they demarcate
the ``flattened soccer ball''.
We explore cylindrical algebraic decompositions, by working through concrete examples.
Maps onto convex polygons and connections to convex optimization are also discussed.}

\begin{figure}[h]
\centering\includegraphics[width=\textwidth]{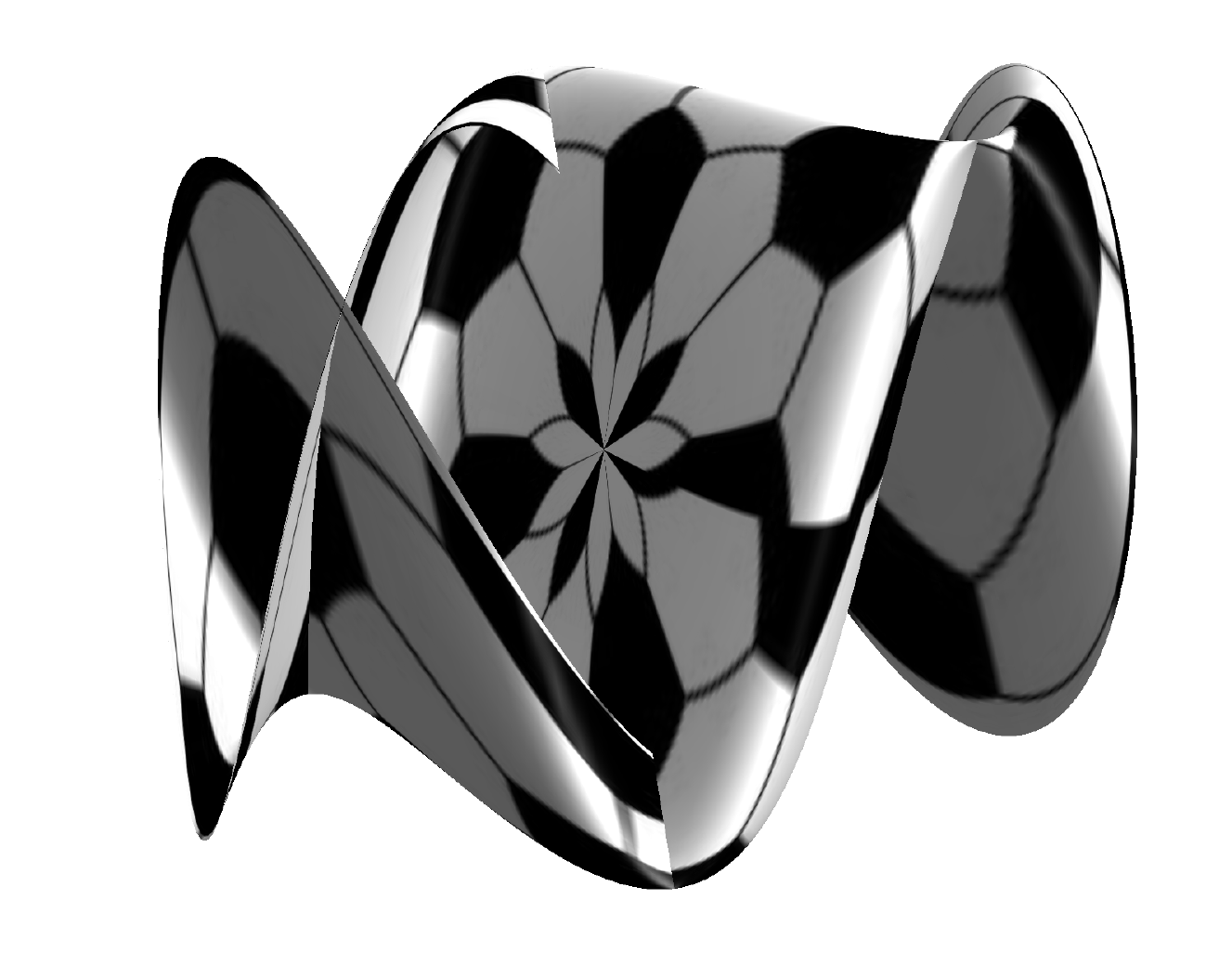}
\caption{\label{fig:truesoccer} A soccer ball is flattened 
and folded into a square.}
\end{figure}

\section{Introduction}
Computational tools  for real algebraic geometry have numerous applications.
This article offers a case study, focused on the following very simple scenario.
We consider a compact semialgebraic subset of real $3$-space that is
defined by one polynomial $h$ in three variables:
\begin{equation}
\label{eq:B}
 \mathcal{B} \,\, = \,\,  \bigl\{\, (u,v,w) \in \R^3 \,:\, h(u,v,w) \geq 0  \,\bigr\}. 
 \end{equation}
We think of $\mathcal{B}$ as our ``soccer ball''. 
A {\em flattening} of $\mathcal{B}$ is its image under a polynomial map
\begin{equation}
\label{eq:phi}
 \phi \,: \,\R^3 \rightarrow \R^2, \,\,(u,v,w) \,\mapsto \, \bigl( \,f(u,v,w) ,\, g(u,v,w)\, \bigr) . 
 \end{equation}
Using quantifiers, the ``flattened soccer ball'' can be expressed as 
$$
 \phi(\mathcal{B}) \, = \,
\bigl\{ (x,y) \in \R^2 \,: \, \, \begin{matrix} \exists\, u,v,w\,:\,
x = f(u,v,w)\,\, \hbox{and} \,\,y = g(u,v,w)  \,\,\hbox{and}  \,\, h(u,v,w) \geq 0 \,\bigr\}.
\end{matrix}  
$$

By Tarski's theorem on quantifier elimination, the image is a semialgebraic set in the plane $\R^2$,
so it can be described as a Boolean combination of polynomial
inequalities. Cylindrical algebraic decomposition \cite{Collins} 
can be used to compute a quantifier-free representation.
This is an active research area and several implementations are available
\cite{BDEMW, Brown, HE, LMX}.
Our aim is to explore the main ingredients in such a representation of
$\phi(\mathcal{B})$.
A  related problem is the computation of the convex hull
${\rm conv}(\phi(\mathcal{B}))$, whose boundary points represent
optimal points for the optimization problem
of maximizing $\lambda f + \mu g $ over $\mathcal{B}$,
where $\lambda, \mu$ are parameters.

This project started in November 2014 at the Simons Institute for the Theory of Computing
in Berkeley, during the workshop 
{\em Symbolic and Numerical Methods for Tensors and Representation Theory}.
The following example was part of its ``Algebraic Fitness Session''.

\begin{figure}[h]
\centering\includegraphics[height=7cm]{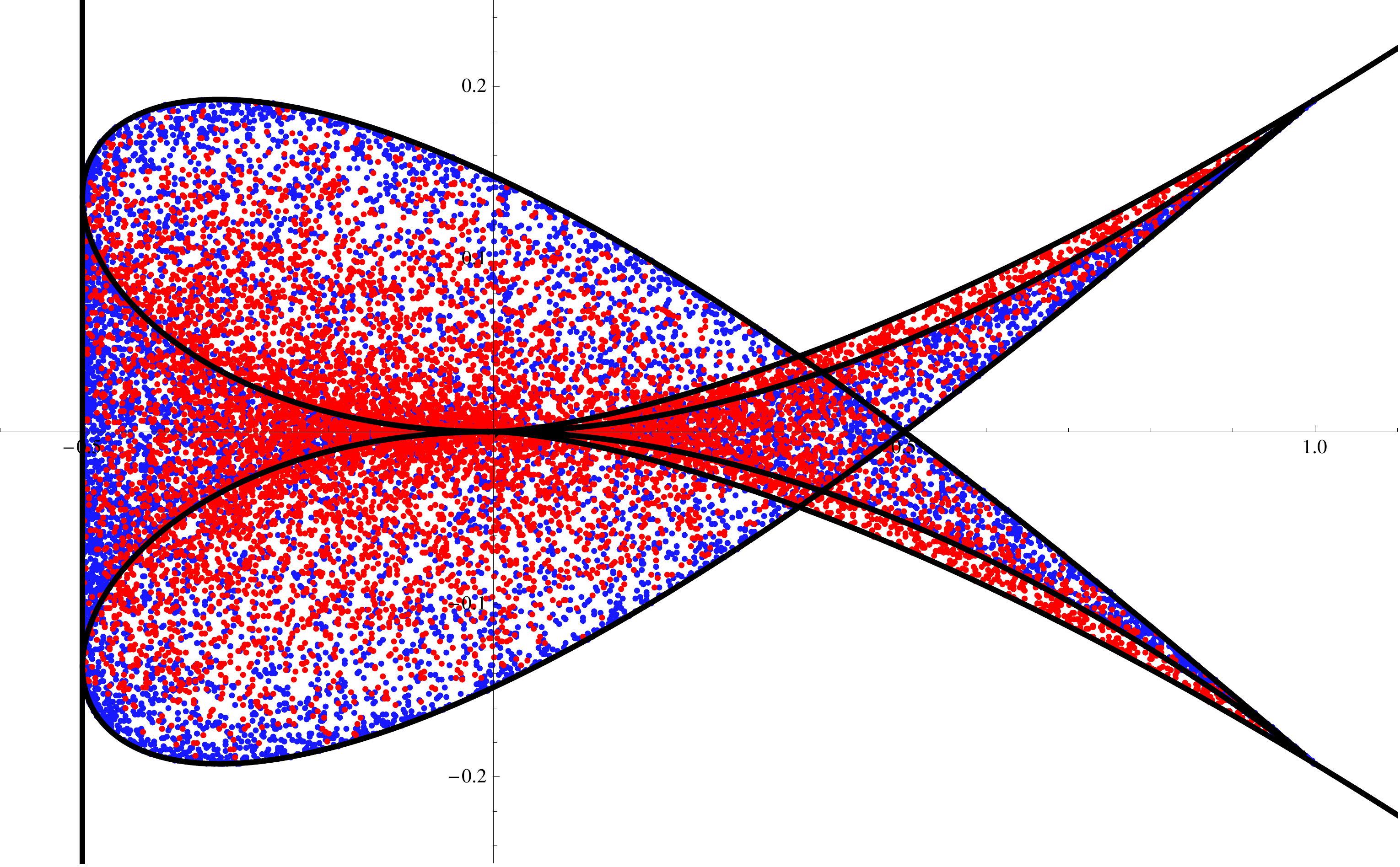}
\caption{Flattening of the unit  ball $\mathcal{B}$ under the map
  $(u,v,w) \mapsto (u v + v w + u w, u v w)$.  Red points are
  randomly sampled from the interior of $\mathcal{B}$, and blue
  points are sampled from the boundary of $\mathcal{B}$.
\label{fig:one}}
\end{figure}

\begin{example}  \label{ex:fitness}
Consider the map given by the two elementary symmetric polynomials,
\[
\phi \,: \,\R^3 \rightarrow \R^2, \,\,(u,v,w) \,\mapsto \, \bigl( \, u v + v w + u w ,\, u v w \, \bigr) . 
\]
We seek to compute the image under $\phi$ of the unit ball
\begin{equation}
\label{eq:unitball}
 \mathcal{B} \,\, = \,\,  \bigl\{\, (u,v,w) \in \R^3 \,:\, u^2+v^2+w^2 \leq 1 \,\,\}.
 \end{equation}
The flattened soccer ball $\phi(\mathcal{B})$  is the 
compact region in $\R^2$ that is depicted in Figure \ref{fig:one}.
In particular, $\phi(\mathcal{B})$ is not convex. If the second coordinate $uvw$
were replaced by a homogeneous quadric then $\phi(\mathcal{B})$ would be convex,
by a theorem of Brickman \cite{Brickman}.


We can quickly get an impression of the flattened ball $\phi(\mathcal{B})$
by sampling points from the ball $\mathcal{B}$
and plotting their images in $\R^2$. These are the red points
in Figure~\ref{fig:one}. We next sample points from the
sphere $\partial \mathcal{B}$ and we plot these in blue.
Figure~\ref{fig:one} shows the existence of two small regions with many
red points but no blue points at all. This means that the image
of the sphere is strictly contained in the image of the ball. In symbols,
$\,\phi(\partial \mathcal{B}) \subset  \phi(\mathcal{B})$.
The Zariski closure of the boundary of  the image $\phi(\mathcal{B})$
is given  by the
 polynomials
$$ p =  x^{3}-27 y^{2} \,\, {\rm and} \,\,
q = (2x{+}1)(4 x^{6}{-}4 x^{5}{-}92 x^{3} y^{2}{+}x^{4}{+}6x^{2} y^{2}
{+}729 y^{4}{+}48 x y^{2}{-}16 y^{2})
$$
Here, $p$ vanishes on the red boundary,
while $q$ vanishes on the blue boundary.
\hfill $\diamondsuit$
\end{example}

For a triple $(f,g,h)$ of polynomials in 
$\R[u,v,w]$, representing the pair $(\mathcal{B},\phi)$,
we define the {\em algebraic boundary} of
$\phi(\mathcal{B})$ to be the Zariski closure in $\C^2$ of
the topological boundary of $\phi(\mathcal{B})$.
In addition to  $\mathcal{B}$ being compact, we  also 
assume that $\mathcal{B}$ is regular, i.e.~the closure of the interior of 
$\mathcal{B}$ contains $\mathcal{B}$. This excludes examples
where lower-dimensional pieces stick out, like  the Whitney umbrella.
With these hypotheses, we can apply results in real algebraic 
geometry, found in~\cite[Lemma~3.1]{KKKR} and~\cite[Lemma~4.2]{Sinn},
to conclude that the algebraic boundary is pure of dimension $1$ in $\C^2$.
It is defined by the product 
of two squarefree polynomials $p$ and $q$ in $\R[x,y]$.
The curve $V(p)$ is the branch locus of the map $\phi$ itself.
It depends only on $f$ and $g$ but not on $h$.
The curve $V(q)$ is the branch locus of the
restriction of $\phi$ to the surface $V(h)$. It depends on $h$.
Note that $q$ is reducible in Example~\ref{ex:fitness}. 

This paper is organized as follows.
In Section~\ref{sec:two} we study the algebraic
geometry underlying our problem.
If the data $f,g,h$ are generic polynomials
then  the curves $V(p)$ and $V(q)$ are irreducible.
We determine their Newton polygons and singularities.
In Section~\ref{sec:three} we explore the
global topology of the flattened soccer ball $\phi(\mathcal{B})$.
We present upper and lower bounds on the
number of connected components in its complement.
Section~\ref{sec:four} introduces tools from symbolic computation for deriving 
an exact representation of $\phi(\mathcal{B})$.
Section~\ref{sec:five} offers 
connections to convexity and to sum-of-squares techniques in polynomial optimization.

\section{Algebraic Curves}
\label{sec:two}

A standard approach in algebraic geometry is to focus
on the generic instance in a family of problems.
This then leads to  an upper bound for
the algebraic complexity of the output that is valid for 
all special instances. In what follows we pursue that standard approach.

Suppose that $f,g$ and $h$ are generic inhomogeneous polynomials
of degrees $d_1,d_2$ and $e$ in $\R[u,v,w]$.
The soccer ball $\mathcal{B}$ and the map $\phi$ are
defined as in (\ref{eq:B}) and~(\ref{eq:phi}).
Let $p$ denote the squarefree polynomial that defines the branch locus of $\phi$,
and let $q$ be the squarefree polynomial that defines the branch locus of $\phi_{\{h=0\}}$.
These polynomials are unique up to scaling. 
They represent the algebraic boundary of $\phi(\mathcal{B})$.
Both curves are in fact irreducible:

\begin{theorem}\label{theorem:generic_degree_soccer}
For generic polynomials $f,g,h$ in $\R[u,v,w]$,  the boundary polynomials $p$ and $q$
of the flattened soccer ball $\phi(\mathcal{B})$ are irreducible. Their Newton polygons are the triangles
$$
\begin{small}
\begin{matrix}
 {\rm Newt}(p) & = & \!\! D_p\cdot {\rm conv}\{(0,0), (0,d_1),(d_2,0)\}
&{\rm where} \!\! &  D_p = d_1^2+d_1d_2+d_2^2-3d_1-3d_2+3;  \smallskip \\
 {\rm Newt}(q) & = & \!\! D_q \cdot {\rm conv}\{(0,0), (0,d_1),(d_2,0)\}  
 &{\rm where} \!\! & \,\,\,\, D_q  =   e (d_1+d_2+e-3) . \qquad \qquad  \qquad
 \end{matrix}
 \end{small}
 $$
 The irreducible complex curves $V(p)$ and $V(q)$ are highly singular,
 with genera
 $$ 
 \begin{small}
 \begin{matrix} 
 {\rm genus}(p) & = &
 \frac{1}{2}(2d_1^3+3d_1^2d_2+3d_1d_2^2+2d_2^3-13d_1^2-16d_1d_2-13d_2^2+27d_1+ 27d_2-20),
\smallskip \\
 {\rm genus}(q) & = & 
 \frac{1}{2}(d_1^2e+2d_1d_2 e+3d_1e^2+d_2^2e+3d_2e^2+2e^3-10d_1e{-}10d_2e-13e^2{+}21e{+}2).
 \end{matrix}
 \end{small}
 $$
 The numbers of singular points of these curves in
 the complex affine plane $\C^2$ are
$$
\begin{small}
\begin{matrix}
 \# {\rm Sing}(V(p)) &=&
 \frac{1}{2}\bigl(\,(D_p \cdot d_1-1) (D_p \cdot d_2-1)\,-\,D_p \cdot {\rm gcd}(d_1,d_2)+1\,\bigr) \,-\, {\rm genus}(p), 
 \smallskip \\
 \# {\rm Sing}(V(q)) & = &
  \frac{1}{2} \bigl(\,(D_q \cdot d_1-1) (D_q \cdot d_2-1)\,-\,D_q \cdot {\rm gcd}(d_1,d_2)+1\,\bigr) \,-\, {\rm genus}(q).
 \end{matrix}
 \end{small}
 $$
\end{theorem}

\smallskip

In this statement, ${\rm genus}(p)$ denotes the genus of the Riemann surface
that is obtained by resolving the singularities of the curve $V(p)$. Equivalently,
this is the {\em geometric genus}. 
The proof of Theorem~\ref{theorem:generic_degree_soccer}
 realizes the plane curves $V(p)$ and $V(q)$ as generic projections
of smooth curves in $3$-space. This implies that all  their singular points
are nodes (cf.~\cite{Joh}), and these are counted by the difference
between the arithmetic genus and the geometric genus.

Table \ref{tab:curves}  underscores how singular our curves are. For instance, the last row concerns
 a general map $\phi$  of degree $4$. The branch locus $V(q)$ of that map restricted
  to the boundary surface $V(h)$
 has degree $56$. A general plane curve of that same degree  has
genus $1485$. However, the genus of our curve $V(q)$ is only $36$, so it
 has $1485-36 = 1449$ singular points.
\begin{table}[h]
$$
\begin{matrix}
(d_1 , d_2 , e) \,& \,{\rm degree}(p) & {\rm genus}(p) & \# {\rm Sing}(V(p)) \,\,&\,\,
 {\rm degree}(q) & {\rm genus}(q) & \# {\rm Sing}(V(q)) \\
 (1, 2, 2) &  2 & 0 & 0 &  8 & 1 & 8 \\
(1, 3, 2) & 12 & 1 & 14 & 18 & 4 & 36 \\
(2, 2, 2) & 6 & 0 & 10 &  12 & 4 & 51 \\
(2, 3, 2) &  21 & 5 & 122 & 24 & 9 & 160 \\
(2, 4, 2) & 52 & 21 & 604 & 40 & 16 & 345 \\
(3, 3, 2) & 36 & 17 & 578 & 30 & 16 & 390 \\
(3, 4, 2) &  76 & 43 & 2048 & 48 & 25 & 792 \\
(4, 4, 2) & 108 & 82 & 5589 & 56 & 36 & 1449 
 \end{matrix}
 \vspace{-0.15in}
 $$
 \caption{\label{tab:curves}  The numerical values in Theorem~\ref{theorem:generic_degree_soccer}
 for input polynomials of low degree.}
 \end{table}

From the polygon ${\rm Newt}(p)$ in Theorem~\ref{theorem:generic_degree_soccer}
we see that the curve $V(p)$ has degree $D_p \cdot {\rm max}(d_1,d_2)$,
and similarly for $V(q)$. When the input polynomials $f,g,h$
of degrees $d_1,d_2,e$ are not generic but special,
these numbers serve as an upper bound. We take the sum of these numbers to get

\begin{corollary}
For any  $f,g,h$, the algebraic boundary of $\phi(\mathcal{B})$ 
has degree at most 
$$ \bigl(d_1^2+d_1d_2+d_2^2-3d_1-3d_2+3 \,+\,  e (d_1+d_2+e-3) \bigr) \cdot {\rm max}(d_1,d_2). $$
This bound is tight when the polynomials $f,g,h$ are generic relative to their  degrees.
\end{corollary}

\begin{remark} \rm
If $ d_1 \leq d_2 = 2$ and $e$ is arbitrary then
the branch curve $V(p)$ of the map $\phi$ has
genus $0$. This  means the curve admits a parametrization by
rational functions.
\end{remark}

The two cases given in the third and fourth row of Table \ref{tab:curves} will be of most interest 
to us. For each of them, we may assume that $\mathcal{B}$ is the 
unit ball (\ref{eq:unitball}), but $\phi$ is arbitrary.

\begin{example}
If we flatten  the unit ball
(\ref{eq:unitball}) via a quadratic map ($d_1 = d_2 = 2$) then
the branch locus of $\phi$ is the rational sextic curve $V(p)$, with
$10$ singular points.  The branch curve of the restriction
of $\phi$ to $V(h)$ is the curve $V(q)$ of degree $12$ and genus $4$,
so it has $51$ singular points. These two curves make up the boundary of $\phi(\mathcal{B})$. 

If both $f$ and $g$ are homogeneous quadrics then the image of 
 $\mathcal{B}$ under $\phi$ is convex. This follows from \cite[Theorem 2.1]{Brickman}. More
precisely, $\phi(\mathcal{B})$ is a spectrahedral shadow, bounded by a
curve of degree six.  This scenario corresponds to the case $p =n=3$
in Table 1 of \cite{SiSt}.
The image $\phi(\mathcal{B})$ is generally not convex when
one of the quadrics $f,g$ is not homogeneous.
For instance, the image of the unit ball under the map
$(u,v) \mapsto ( u^2-v, v^2)$ is not convex.
\hfill $\diamondsuit$
\end{example}

\begin{example}
Let $d_1 = 2, d_2 =3$ as in Example \ref{ex:fitness},
but with $f$ and $g$ generic. The picture of
 $\phi(\mathcal{B})$ is now much more complicated than that 
  in Figure \ref{fig:one}. The red boundary $V(p)$
 is a curve of degree $21$ with $122$ complex singular points,
 and the blue boundary $V(q)$ is a curve of degree $24$ with
 $160$ complex singular points.
 This is worked out in Example~\ref{example:generic_example_algorithm}.
 \hfill $\diamondsuit$ 
\end{example}

\begin{proof}[Proof of Theorem~\ref{theorem:generic_degree_soccer}]
We consider two curves in affine $3$-space $\C^3$.
The curve $C_1$ is defined by the $2 \times 2$-minors
of the Jacobian matrix of $(f,g)$ with respect to $(u,v,w)$.
This $2 \times 3$-matrix has general entries
of degree $d_1-1$ in the first row and general entries of
degree $d_2-1$ in the second row. By the Thom-Porteous-Giambelli Formula,
we have ${\rm deg}(C_1) = (d_1{-}1)^2 + (d_1{-}1)(d_2{-}1) + (d_2{-}1)^2 $.
This expression equals $D_p$.
 The curve $C_2$ is the complete intersection defined
by the polynomial $h$, which has degree $e$,
and the Jacobian determinant of $(f,g,h)$ with respect to $(u,v,w)$,
 which has degree $d_1+d_2+e-3$.
By B\'ezout's Theorem, ${\rm deg}(C_2) = e (d_1+d_2+e-3)$. The hypothesis that $f,g$ and $h$
are generic ensure that  $C_1$ and $C_2$ are smooth and irreducible.
Their degrees are the quantities $D_p$ and $D_q$ in the
statement.

Both of the results from algebraic geometry that were used in the previous paragraph
(Thom-Porteous-Giambelli and B\'ezout)
require certain genericity hypotheses on the geometric data
to which they apply. These hypotheses are satisfied in our case
because the given polynomials $f$, $g$ and $h$
are assumed to have generic coefficients.
See e.g.~\cite[Section 3.5.4]{Manivel}.

The curves defined by $p$ and $q$ are 
the images of $C_1$ and $C_2$ 
under the map $\phi = (f,g)$
from $\C^3$ to $\C^2$.
Our first claim states that, for $i=1,2$, the Newton polygon
of the plane curve is the triangle
$r \cdot {\rm conv}\{(0,0), (0,d_1),(d_2,0)\}  $,
where $r = {\rm deg}(C_i)$.

We prove this using {\em tropical geometry} \cite{MS}.
 By genericity of $f,g$ and $h$,
the tropical curve ${\rm trop}(C_i)$ in $\R^3$ is the $1$-dimensional fan
with rays spanned by $(1,0,0)$, $(0,1,0)$, $(0,0,1)$ and $(-1,-1,-1)$,
where each ray has multiplicity $r$.
Our goal is to compute the tropical curve ${\rm trop}(\phi(C_i))$ in $\R^2$.
This contains the image of ${\rm trop}(C_i)$ under the
tropicalization of the map $\phi$. This is the piecewise-linear map
that takes $(U,V,W)$ in $\R^3$ to 
$\,\bigl({\rm min} \{d_1 U, d_1 V, d_1 W, 0\},
{\rm min}\{ d_2 U, d_2 V, d_2 W, 0 \} \bigr)$. Its image
is the weighted ray in $\R^2$ spanned by $(-d_1 r ,-d_2 r)$. The other rays
of the tropical curve ${\rm trop}(\phi(C_i))$ arise from the
points of $C_i$ at which $f$ and $g$ vanish. We derive these
using the method of Geometric Tropicalization, specifically
\cite[Theorem 6.5.11]{MS}. The relevant very affine curve
is $C_i \backslash \{uvw fg = 0 \}$, and the normal crossing 
boundary in the SNC pair is the divisor defined by $uvwfg$ on $C_i$.

The surface $\{f = 0\}$ meets the curve $C_i$ in $ d_1 r$ points,
and the divisorial valuations at these points map to the 
weighted ray $(d_1 r,0)$ in $\R^2$. Likewise, the surface $\{g=0\}$ meets $C_i$
in $d_2r$ points, and their divisorial valuations create the 
weighted ray $(0,d_2 r)$ in $\R^2$. Hence the tropical plane curve
${\rm trop}(\phi(C_i))$ consists of the three weighted rays specified by
$(-d_1 r ,-d_2 r)$, $(d_1r, 0)$ and $(0,d_2 r)$.
This implies our assertion about the Newton polygons of $p$ and $q$.

To prove the second assertion, 
about the genera of the two curves in question,
we use the following two facts about general curves in $\PP^3$.
These are easily derived by computing the Hilbert series and then reading off the Hilbert polynomial.
Recall that, for a curve with the Hilbert polynomial $h(n)=h_1n+h_0$, the degree is $h_1$ and the arithmetic genus is $1-h_0$. Moreover, if the curve is smooth, then its geometric genus equals the arithmetic genus.
\begin{itemize}
\item A smooth space curve defined by the $2 \times 2$-minors
of a $2 \times 3$-matrix with rows of degrees $a$ and $b$ has
degree $a^2+ab+b^2$ and genus
$a^3+\frac{3}{2}a^2b+\frac{3}{2}ab^2+b^3-2a^2-2ab-2b^2+1$.
\item The complete intersection of two general surfaces of degrees $a$ and $b$
 in $\PP^3$ is a smooth curve of
degree $ab$ and  genus $\frac{1}{2}ab(a+b-4)+1$.
\end{itemize}

The genus of the plane curve $V(p)$ is equal to the genus of the space curve
$C_1$ that maps to it, and similarly for $V(q)$ and $C_2$. So, it suffices
to compute the genera of the affine curves $C_1$ and $C_2$ in $\C^3$.
We may work with their projective closures $\overline{C}_1$ and $\overline{C}_2$
in $\PP^3$. The curve $\overline{C}_1$ has the
determinantal representation as in the first bullet, with 
$a = d_1-1$ and $b=d_2-1$. Substitution yields the desired formula for ${\rm genus}(p)$.
 The curve $\overline{C}_2$ is the complete
intersection of two surfaces in $\PP^3$, of degree $a = e$ and $b = d_1+d_2+e-3$.
Substituting these expressions into $\frac{1}{2}ab(a+b-4)+1$, we obtain the desired
formula for ${\rm genus}(q)$. 

We can regard $V(p)$ and $V(q)$
as curves in the weighted projective plane
given  by the known Newton polygons.
 The genus of a general curve
of the same degree is the number of interior lattice points on the Newton triangle.
That number is equal~to
\begin{small}
$$
\# \bigl(\mathbb{Z}^2 \cap {\rm int}({\rm conv}\{(0,0),(0,rd_1),(rd_2,0)\} ) \bigr) \,=\,
\frac{1}{2} \bigl(\,(r d_1-1) (r  d_2-1)\,-\, {\rm gcd}(rd_1,rd_2)+1\,\bigr).
$$ \end{small} 
Here $r$ is $D_p$ or $D_q$ as before. The number of singular points is
the number above minus the  genus of the curve.
This gives the count in 
the last assertion of Theorem~\ref{theorem:generic_degree_soccer}.
\end{proof}

We used the computer algebra system {\tt Macaulay2}  \cite{M2}
to verify some of the entries in Table~\ref{tab:curves}.
Here is the {\tt Macaulay2} code we used for a typical computation with $d_1 = d_2 = e = 2$:
\begin{verbatim}
S = QQ[x,y,u,v,w];  h = u^2+v^2+w^2-1;
f = u*v-u*w+7*v^2+v*w+5*w^2+u+v+w+1;
g = u^2-u*v+u*w-v^2+v*w-w^2+u-v+w-1;
C1 = minors(2,jacobian(ideal(f,g)));
C2 = minors(3,jacobian(ideal(f,g,h)))+ideal(h);
p = first first entries gens 
                 eliminate({u,v,w},C1+ideal(x-f,y-g))
Ip = radical(ideal(diff(x,p),diff(y,p),p));
{degree p, # terms  p, degree Ip}
q = first first entries gens 
                 eliminate({u,v,w},C2+ideal(x-f,y-g))
Iq = radical(ideal(diff(x,q),diff(y,q),q));
{degree q, # terms  q, degree Iq}
\end{verbatim}
The polynomials {\tt p} and {\tt q} have
 degrees $6$ and $12$ respectively.
 The command {\tt \# terms} verifies 
 that all monomials in the Newton polygons appear
 with non-zero coefficients.
The singular loci of the two curves
are given by their radical ideals, {\tt Ip} and {\tt Iq}.
Applying the command {\tt degree} to these ideals verifies that
the number of singular points is $10$ and $51$ respectively.

\section{Topological Complexity}
\label{sec:three}

When a soccer ball gets flattened, one generally expects the planar image to be
simply connected. However, it is quite possible for $\phi(\mathcal{B})$ to have holes.
In other words, the complement $\R^2 \backslash \phi(\mathcal{B})$
can have two or more connected components. In this section we present
an explicit construction that makes this happen, 
with the number of holes being arbitrarily large.

The number of connected components of $\phi(\mathcal{B})$ is at most
the number of connected components of $\mathcal{B}$. 
The number of its holes is counted by the first Betti number of $\phi(\mathcal{B})$.
The best upper bounds for Betti numbers of
 compact semialgebraic sets are due to Basu and Riener~\cite[Theorem 10]{BR2} and 
 Basu and Rizzie~\cite[Theorem 27]{BR}. 
In our setting,  the number of holes is bounded by
   $O( \max(d_1,d_2)^6 e^2)$ if $e \leq \max(d_1,d_2)$ and by $O( \max(d_1,d_2)^8)$ otherwise.

In what follows we assume that $\mathcal{B}$ is the unit ball (\ref{eq:unitball}).
 The image $\phi(\mathcal{B})$ is a compact connected subset of $\R^2$.
We are interested in maps $\phi$ whose image $\phi(\mathcal{B})$ is not simply connected.
The construction we shall give furnishes   
   the lower bound $O(d_1 d_2)$ on the number of holes of $\phi(\mathcal{B})$. Based on
   {\em Lissajous curves}, it  gives rise to some  beautiful explicit examples.

The {\em Chebyshev polynomials (of the first kind)}
are defined recursively by 
$$ T_0(t)\,=\,1,\,\, T_1(t)\,=\,t \,\,\,
{\rm and} \,\,\, T_{d+1}(t)\,=\,2tT_d(t)-T_{d-1}(t) \quad
\hbox{for $d \geq 1$}. $$
Explicitly, the Chebyshev polynomials are
$$ T_2(t) = 2 t^2 - 1, \,\, T_3(t) = 4t^3 - 3 t, \,\, T_4(t) = 8 t^4 -  8 t^2 + 1 , \,\, 
T_5(t) = 16 t^5 - 20 t^3 + 5 t ,\,\,\ldots . $$
They satisfy the trigonometric identity
 $\,  {\rm cos}(d \theta) = T_d( {\rm cos}(\theta))$.
 Fix relatively prime positive integers $d_1$ and $d_2$ with $d_1 < d_2$.
Let $\mathcal{L}_{d_1,d_2}$ denote the  Lissajous curve 
\begin{equation}
\label{eq:trigpara}
x = {\rm cos}(d_1 \theta) \,, \,\,\, y = {\rm cos}(d_2 \theta). 
\end{equation}
Its Zariski closure is the curve of degree $d_2$ with polynomial parametrization
$$ x = T_{d_1}(t)\,,\,\, y = T_{d_2}(t). $$
For instance, Lissajous curve  $\mathcal{L}_{2,3}$ is the rational cubic 
$\{4 x^3 - 2y^2 - 3 x + 1 = 0\}$. It is singular at
$(x,y) = \bigl(\frac{1}{2},0 \bigr)$.

\begin{example}
Figure \ref{figure:Lissajous57} shows the Lissajous curve $\mathcal{L}_{5,7}$.
This curve has $12$ singular points. This is the number of bounded regions
in the complement of $\mathcal{L}_{5,7}$ in $\R^2$.
\end{example}

\begin{figure}[h]
\centering
\includegraphics[width=0.4\textwidth]{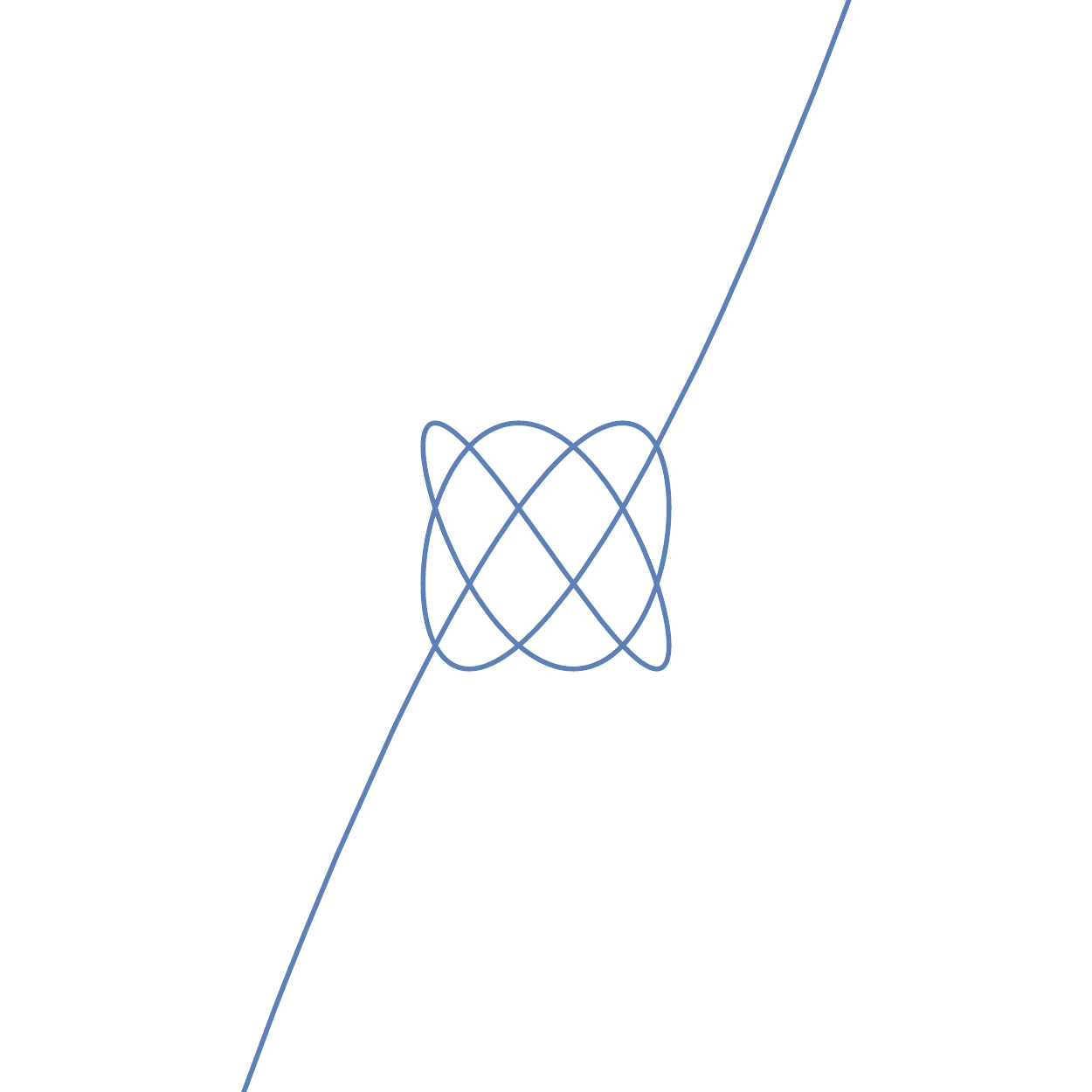}
\caption{The Lissajous curve $\mathcal{L}_{5,7}$. \label{figure:Lissajous57}}
\end{figure}

\begin{lemma}
The curve $\mathcal{L}_{d_1,d_2}$ has
precisely $\frac{(d_1-1)(d_2-1)}{2}$ complex singular points.
All of these are real and are attained by two distinct
values of $\,\theta$ in the
trigonometric parametrization~(\ref{eq:trigpara}).
\end{lemma}

\begin{proof}
By the same argument as in the proof of Theorem~\ref{theorem:generic_degree_soccer},
the Newton polygon of the  Lissajous curve $\mathcal{L}_{d_1,d_2}$
is contained in the triangle with vertices $(0,0)$, $(d_2,0)$ and $(0,d_1)$.
The number of interior lattice points of that triangle is
 $\frac{(d_1-1)(d_2-1)}{2}$. This is the genus of the 
 generic curve with that Newton polygon. And,     it hence is an upper bound on the number
 of complex singular points of the special curve   $\mathcal{L}_{d_1,d_2}$.
 
We next exhibit $\frac{(d_1-1)(d_2-1)}{2}$ real singular points
on $\mathcal{L}_{d_1,d_2}$ that are in the image of~(\ref{eq:trigpara}).
Pick any $k \in \{ 1,\ldots,d_1-1 \}$ and any $l \in \{1,\ldots,d_2-1\}$.
Consider the angles
\begin{equation}
\label{eq:thetaprime}
\theta'=\pi(\frac{k}{d_1} + \frac{l}{d_2})\quad \hbox{and} \quad
\theta''=\left|\pi(\frac{l}{d_2}-\frac{k}{d_1})\right|.
\end{equation}
If $\frac{l}{d_2}-\frac{k}{d_1}>0$, then $\theta' - \theta'' = \frac{2k\pi}{d_1}$ and
$\theta' + \theta'' = \frac{2l \pi}{d_2}$; otherwise
$\theta' - \theta'' = \frac{2l \pi}{d_2}$ and
$\theta' + \theta'' = \frac{2k\pi}{d_1}$.
This means that $\theta'$ and $\theta''$ map to the same
point, and the Lissajous curve $\mathcal{L}_{d_1,d_2}$ has a node
at that point.
There are $(d_1-1)(d_2-1)$ choices of pairs $(k,l)$.
Since the trigonometric parametrization~(\ref{eq:trigpara}) is
$2$-to-$1$ on the interval $[0,2\pi]$, this creates 
$\frac{(d_1-1)(d_2-1)}{2}$  nodal singularities on $\mathcal{L}_{d_1,d_2}$.
This argument is a modification of~\cite[Section~2.1]{BHJS}.
The lower bound we derived matches the upper bound in the previous paragraph,
and this completes the proof.
\end{proof}

We now apply this to flattening the soccer ball.
Consider the map  $\phi = (f,g)$ with
\begin{equation}
f(u,v,w) = T_{d_1}(u) + \epsilon \cdot v ,\qquad
g(u,v,w) = T_{d_2}(u) + \epsilon \cdot w, 
\label{eq:lissajous}
\end{equation}
where $T_d(\cdot)$ is the degree-$d$ Chebyshev polynomial,
  and $\epsilon>0$ is a small constant.
The map $\phi$  takes the soccer ball $\mathcal{B}$ and creates
a two-dimensional image with many holes in $\R^2$.

\begin{example} \label{ex:zweidreiliss} 
Let $d_1 = 2$ and $d_2 = 3$. The set $\phi(\mathcal{B})$
is the region shown in Figure~\ref{figure:Lissajous23}. It has precisely one hole.
This picture was created by the following code in {\tt Mathematica},
which produces a huge expression:
\begin{small}
\begin{verbatim}
h = 1 - (u^2 + v^2 + w^2);
f = 2*u^2 - 1 + 1/10*v;  g = 4*u^3 - 3*u + 1/10*w;
S = Exists[{u, v, w}, h >= 0 && x == f && y == g];
SR = Resolve[S, Reals]
RegionPlot[SR, {x, -1.2, 1.2}, {y, -1.2, 1.2}, PlotPoints -> 50]
\end{verbatim}
\end{small}
The command ``{\tt Resolve}'' performs quantifier elimination.
\end{example}

\begin{figure}[h]
\centering
\includegraphics[width=0.4\textwidth]{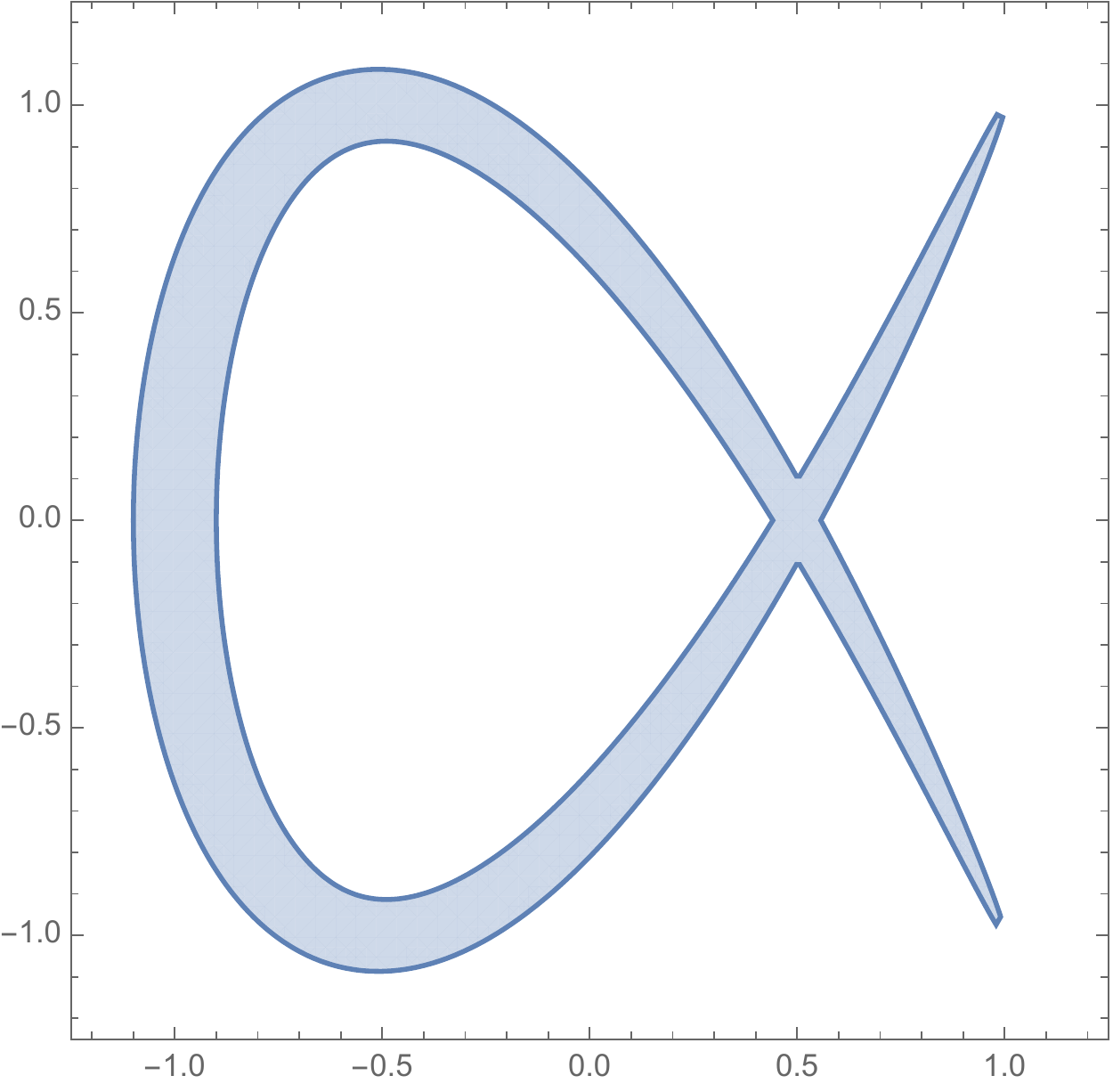}
\caption{The modification of the Lissajous curve in Example~\ref{ex:zweidreiliss}.}\label{figure:Lissajous23}
\end{figure}

The following is our main result in this section. 

\begin{theorem}
Let $d_1<d_2$ be relatively prime 
and $\phi$ as above with
$\epsilon>0$ sufficiently
small. Then, the complement of $\,\phi(\mathcal{B})\,$ in $\R^2$ has
$\frac{(d_1-1)(d_2-1)}{2}+1$ connected components.
The algebraic boundary of $\phi(\mathcal{B})$ is an irreducible curve
of degree at most $4d_2-2$. It is the branch locus of $\phi_{\{h=0\}}$, so it is
defined by the polynomial that was denoted by $q$ in Theorem~\ref{theorem:generic_degree_soccer}.
\end{theorem}

\begin{proof}
The part of the curve $\mathcal{L}_{d_1,d_2}$
that lies in the square $[-1,1]^2$ is compact.
We regard this as an embedded planar graph,
where the vertices are the nodal singularities
given in~(\ref{eq:thetaprime}) together with the two degree $1$ endpoints, and the edges
are the pieces of the Lissajous  curve that connect the nodes and endpoints.
This planar graph is $4$-valent, the numbers
of vertices, edges and faces satisfy
$v-e+f = 2$ and $2 e = 4(v-2)+2$. 
This implies $ f = v-1 = \frac{(d_1-1)(d_2-1)}{2}+1$,
i.e.~the Lissajous curve has the correct number of holes.

As $\epsilon $ increases from $0$ to being positive,
the curve gets replaced by a two-dimensional region. But
the number of holes in the complement does not change.

The algebraic boundary of $\phi(\mathcal{B})$ is given by the
polynomials $p$ and $q$ that describe the branch curves of
$\phi$ and $\phi_{\{h=0\}}$ respectively. However, in the present case, the curve $V(p)$
does not exist because the Jacobian of the map $\phi$
has rank $2$ for all $(u,v,w) \in \C^3$.
The Jacobian determinant of $(f,g,h)$ with respect to $(u,v,w)$ is
the irreducible polynomial
\begin{equation}
\label{eq:Jacodet}
 {\rm det} \begin{pmatrix}
\, \partial T_{d_1}/\partial u & \epsilon & 0 \\
\, \partial T_{d_2}/\partial u & 0 & \epsilon \\
-2 u & -2v & -2w \,\\
\end{pmatrix} \quad = \quad 
2v \epsilon \frac{\partial T_{d_1}}{\partial u} + 
2w \epsilon \frac{\partial T_{d_2}}{\partial u} - 2 \epsilon^2 u.
\end{equation}
This is a polynomial of degree $d_2$ in which 
 $v$ and $w$ occur linearly. The intersection of 
this surface with the unit sphere is an irreducible curve of degree at most $2d_2$. 
To compute the image of the curve, we substitute
 $\,v = \frac{1}{\epsilon}(x-T_{d_1}(u)) \,$ and
 $\,w = \frac{1}{\epsilon}(y-T_{d_2}(u))\,$ into
 $h(u,v,w)$ and into (\ref{eq:Jacodet}).
 This results in two polynomials in $u,x,y$.
Our  task is to eliminate $u$. We do this by taking the determinant of the
 Sylvester matrix with respect to $u$. The non-constant entries in the Sylvester matrix
have degree one or two in $x$ or $y$. By examining their pattern in the matrix, 
we find that the determinant is a polynomial of degree at most $4d_2-2$.
\end{proof}

\begin{remark}
We found experimentally that the Newton polygon of $q$ is the triangle with vertices
$(0,0)$, $(4d_2-2,0)$ and $(0,2d_1+2d_2-2)$, but we could not prove this.
\end{remark}

\begin{example} \label{ex:fitnessMathematica} 
We return to the flattened soccer ball seen in the introduction.
To draw this picture from scratch in {\tt Mathematica}, we run
the code in Example~\ref{ex:zweidreiliss}, modified as follows:
\begin{verbatim}
f = u*v + v*w + u*w;  g = u*v*w;
\end{verbatim}
For this input, the output of the quantifier elimination command {\tt Resolve} equals:
\begin{footnotesize}
\begin{verbatim}
(-(1/2) <= x <= 0 && y == 0) || 
   (y == -(1/(3*Sqrt[3])) && x == -(1/3)) || 
   (-(1/(3*Sqrt[3])) < y < 0 && 
      Root[2*#1^3 + #1^2 - y^2 & , 1] <= x <= 
        Root[2*#1^3 + #1^2 - y^2 & , 2]) || 
   (y == 0 && Inequality[-(1/2), LessEqual, x, Less, 
        0]) || (0 < y < 1/(3*Sqrt[3]) && 
      Root[2*#1^3 + #1^2 - y^2 & , 1] <= x <= 
        Root[2*#1^3 + #1^2 - y^2 & , 2]) || 
   (y == 1/(3*Sqrt[3]) && x == -(1/3)) || 
   (x == -(1/2) && -(1/(3*Sqrt[6])) <= y <= 
        1/(3*Sqrt[6])) || (-(1/2) < x < -(1/3) && 
      Root[729*#1^4 + #1^2*(-(92*x^3) + 6*x^2 + 48*x - 
                   16) + 4*x^6 - 4*x^5 + x^4 & , 1] <= y <= 
        Root[729*#1^4 + #1^2*(-(92*x^3) + 6*x^2 + 48*x - 
                   16) + 4*x^6 - 4*x^5 + x^4 & , 4]) || 
   (x == -(1/3) && -(1/(3*Sqrt[3])) <= y <= 
        1/(3*Sqrt[3])) || 
   (-(1/3) < x < (1/38)*(5*Sqrt[5] - 7) && 
      Root[729*#1^4 + #1^2*(-(92*x^3) + 6*x^2 + 48*x - 
                   16) + 4*x^6 - 4*x^5 + x^4 & , 1] <= y <= 
        Root[729*#1^4 + #1^2*(-(92*x^3) + 6*x^2 + 48*x - 
                   16) + 4*x^6 - 4*x^5 + x^4 & , 4]) || 
   (x == (1/38)*(5*Sqrt[5] - 7) && -Sqrt[2*x^3 + x^2] <= 
        y <= Sqrt[2*x^3 + x^2]) || 
   ((1/38)*(5*Sqrt[5] - 7) < x < 16/43 && 
      Root[729*#1^4 + #1^2*(-(92*x^3) + 6*x^2 + 48*x - 
                   16) + 4*x^6 - 4*x^5 + x^4 & , 1] <= y <= 
        Root[729*#1^4 + #1^2*(-(92*x^3) + 6*x^2 + 48*x - 
                   16) + 4*x^6 - 4*x^5 + x^4 & , 4]) || 
   (16/43 <= x <= 1/2 && -(Sqrt[x^3]/(3*Sqrt[3])) <= 
        y <= Sqrt[x^3]/(3*Sqrt[3])) || 
   (1/2 < x < 1 && (-(Sqrt[x^3]/(3*Sqrt[3])) <= y <= 
           Root[729*#1^4 + #1^2*(-(92*x^3) + 6*x^2 + 48*x - 
                      16) + 4*x^6 - 4*x^5 + x^4 & , 2] || 
         Root[729*#1^4 + #1^2*(-(92*x^3) + 6*x^2 + 48*x - 
                      16) + 4*x^6 - 4*x^5 + x^4 & , 3] <= y <= 
           Sqrt[x^3]/(3*Sqrt[3]))) || 
   (x == 1 && (y == -(1/(3*Sqrt[3])) || 
         y == 1/(3*Sqrt[3])))
\end{verbatim}
\end{footnotesize}
This is a quantifier-free formula for
the flattened soccer ball in Figure~\ref{fig:one}.
Most readers will find such an output hard to understand.
The next section offers an alternative.
\end{example}

\section{Exact Representation of the Image}
\label{sec:four}

Quantifier elimination for polynomial systems over $\R$
is usually performed by cylindrical algebraic decomposition~\cite{Collins}, abbreviated CAD.
Many variants can be found in the recent literature,
including {\em truth table invariant CAD} \cite{BDEMW}
and {\em variant quantifier elimination}  \cite{HE}.
CAD represents a semialgebraic set
as a union of cells. In dimension one this would be
 a disjoint union of points and open intervals.
 Several implementations of CAD are now available,
 including {\tt QEPCAD}~\cite{Brown}, and the packages
  {\tt RegularChains}~\cite{LMX} and {\tt ProjectionCAD}~\cite{EWBD} in {\tt Maple}.
  In Example~\ref{ex:fitnessMathematica},
 we experimented with the implementation of CAD  in {\tt Mathematica}.
 
This section is purely expository, aimed at all
mathematicians and their students. We show how to obtain
a meaningful CAD ``by hand'' for all instances with parameters $d_1=2,d_2=3,e=2$.
 Experts and CAD developers might find this useful as a family of test cases.

Consider the curve in $\R^2$  defined by the polynomial $\,p\cdot q$.
Our image $\phi(\mathcal{B})$ is the closure of
a union of connected components of its complement.
We compute a partition of $\R^2$ that refines the partition
given by $V(pq)$. A key step is to label each
open piece in the finer partition.
We then test which pieces lie in $\phi(\mathcal{B})$,
and we report the labels of those that do.

\begin{algorithm}[h!]
\caption{$\,\,$ Quantifier-free representation of the flattened soccer ball
$\phi(\mathcal{B})$}\label{algorithm:semialgebraic_sampling}
\begin{algorithmic}
\State {\bf Input}: Three polynomials $f,g,h \in \mathbb{Q}[u,v,w]$.
\State           \textit{\textbf{Step 1}: Compute the polynomial $p$ that defines the branch locus of the map $\phi$.}
\State           \textit{\textbf{Step 2}: Compute the polynomial $q$ that defines the branch locus of the restriction of $\phi$ to the boundary surface $V(h)$ of $\mathcal{B}$.} 
\State  \textit{\textbf{Step 3}: Compute the
sorted  list $C$ of all critical $x$-values of the polynomial $p\cdot q$.}
\State  \textit{\textbf{Step 4}: Sample points uniformly from $\mathcal{B}$ and compute their images under the map $\phi$. Save the result in a list $S$.} 
\State  \textit{\textbf{Step 5}: For each $(\overline{x},\overline{y}) \in S$
determine $k \in \mathbb{N}$ such that $\overline{x}$ is between the $k$-th and $(k+1)$-st element of $C$. Compute sorted list $R$ of all real roots
of the univariate polynomial $p(\overline{x},y) \cdot q(\overline{x},y)$.
 Determine $l \in \mathbb{N}$ such that $\overline{y}$ is between the $l$-th and $(l+1)$-st root in $R$. 
 \smallskip }
\State  {\textbf{Output}: The polynomials $p,q$,
the list $C$, and the set of all pairs $(k,l)$ generated in Step~5.}
\end{algorithmic}
\end{algorithm}

Algorithm~\ref{algorithm:semialgebraic_sampling} describes what we do.
The geometric idea is to project $V(pq)$ onto the $x$-axis.
The critical points of that projection come in four flavors:
singular points of $V(p)$, singular points of $V(q)$, 
points in the intersection $V(p,q)$, and smooth points
on $V(pq)$ with vertical tangent lines.
The {\em critical $x$-values} are the $x$-coordinates of all real critical points.

Two consecutive critical $x$-values define a vertical strip.
Here, the behavior of the curve segments  does not change as $x$ varies.
In particular, curve segments do not cross or change direction. 
Curve segments pass from the left to the right over two consecutive
critical $x$-values, and they divide the vertical strip into open regions.
Two of them  are unbounded and hence irrelevant.
  Each bounded region is either  contained
in $\phi(\mathcal{B})$ or is disjoint from $\phi(\mathcal{B})$.

Our description of  $\phi(\mathcal{B})$ has three parts:
the polynomials $p$ and $q$ that define the algebraic boundary,
 the critical $x$-values, and a set of pairs of positive integers. 
 A pair $(k,l)$ determines a region in $\R^2$ as follows.
 The $x$-coordinates are between the  $k$-th and $(k+1)$-st critical $x$-value,
 and the $y$-coordinates lie between $l$-th and $(l+1)$-st curve segment in the $y$-direction.

Steps 1 and 2 of Algorithm~\ref{algorithm:semialgebraic_sampling} 
 can be done with
  \texttt{Macaulay2}, as shown  at the end of Section~\ref{sec:two}. 
  Step 3 is more delicate because $p$ and $q$ are fairly large,
  even when $f,g,h$ are small. The task is to compute
  the critical $x$-values of the polynomials $p$ and $q$, and
  also the $x$-coordinates of the common zeros of $p$ and $q$. 
  To find these $x$-values symbolically,
   we compute three resultants of two polynomials in
   $(x,y)$ with respect to $y$. Namely, we compute
\begin{equation}  
\label{eq:mustcompute}
{\rm Res}_y \bigl(p,\partial{p}/\partial{y}\bigr),\,\,
{\rm Res}_y \bigl(q,\partial{q}/\partial{y}\bigr)\,\,\, \hbox{and} \,\,\,
{\rm Res}_y \bigl(p,q \bigr).
\end{equation}

At this stage one might compute the real roots of these three univariate polynomials
in $x$. This can be done numerically using various methods, including the
numerical algebraic geometry package in {\tt Macaulay2}. However, we
did not  do this in our computations. Instead,
we identify the real roots of our three resultants purely symbolically,
using the command \texttt{Solve} with the option \texttt{Reals} in \texttt{Mathematica}. 
\texttt{Solve} tries to write each solution explicitly in terms of radicals, and if
unsuccessful, it creates a representation as a root of a polynomial.

Since $\mathcal{B}$ is compact, we can enclose it inside an appropriate cube in $\R^3$.
We sample points $(u,v,w)$ with rational coordinates from that cube,
and we throw out points that do not satisfy $h(u,v,w) \geq 0$. For instance, if $\mathcal{B}$ 
is the unit sphere, then we first uniformly sample points from the cube $[-1,1]^3$ and
 keep the points that satisfy $1-u^2-v^2-w^2 \geq 0$. 

Step 4 is  probabilistic. As the number of samples grows to infinity,
every region in the CAD of $\phi(\mathcal{B})$ will be reached and certified eventually.
For any finite number of samples, there is a positive probability that some
region is missed.
Our approach can 
be turned this into a deterministic method by selecting a sample point in each
region and deciding whether it has a preimage in $\mathcal{B}$. This 
amounts to testing whether a semialgebraic set in $\R^3$ is non-empty.
The best known algorithm for deciding non-emptiness of a semialgebraic set is by Renegar~\cite{Renegar}.
 In practice this can be done using implementations of CAD, but we decided not to pursue 
the deterministic variant in the present study.

In our examples we are also interested in deciding which regions of
$\phi(\mathcal{B})$ lie in $\phi(\partial \mathcal{B})$.  Note the distinction between
the colors in Figure \ref{fig:one}. For instance,
to obtain points on the boundary of the unit ball, we uniformly sample points from the square $[-1,1]^2$, keep the points that satisfy $1-u^2-v^2 \geq 0$, and create two boundary points with $w=\pm \sqrt{1-u^2-v^2}$. 

In Step 5, we use binary search to determine $k$ such that $\overline{x}$ 
is between the $k$-th and $(k+1)$-st element of 
the list $C$. The symbolic representation derived in
 \texttt{Mathematica} is suitable for doing this. To be precise, we did the binary search
 with the following commands:
\begin{verbatim}
Block[{$ContextPath}, Needs["Combinatorica`"]]
k = Combinatorica`BinarySearch[C, x]-1/2
\end{verbatim}
The real roots of $p(\overline{x},y) \cdot q(\overline{x},y)$ can again be computed using \texttt{Mathematica} command \texttt{Solve}.
We now illustrate Algorithm \ref{algorithm:semialgebraic_sampling}
by applying it to the instance that launched this project.

\begin{example}\label{example:fitness_example_with_our_algorithm} 
Fix $f,g,h$ as in Example \ref{ex:fitness}. 
The first part of the output are the polynomials $p$ and $q$ at the end of
Example \ref{ex:fitness}. The second part is the list of critical $x$-values:
 \begin{equation}
 \label{eq:criticalvalues}
 -\frac{1}{2}, \,0, \,\frac{16}{43}, \,\frac{2}{5}, \,\frac{1}{2}, \,1 .
 \end{equation}
 The third part is a list which 
 represents a partition of $\phi(\mathcal{B})$  into $22$ regions:
 \begin{align*}
&(1, 1), (1, 2), (1, 3), 
(2, 1), (2, 2), (2,3), (2, 4), (2, 5), 
(3, 1), (3, 2), (3,3),\\& (3, 4), (3, 5), 
(4, 1), (4, 2), (4,3), (4, 4), (4, 5), 
(5, 1), (5, 2), (5,4), (5, 5). 
\end{align*}
The output above represents a quantifier-free formula
for the flattened soccer ball. Using symbolic computation,
we can assign one of the $22$ labels to any sample point
$\phi(u_0,v_0,w_0)$. For instance, 
$\phi(\mathcal{B}) \backslash \phi(\partial \mathcal{B})$
consists of the six regions $(3,1),(3,5),(4,1),(4,5), (5,1),(5,5)$.

\begin{figure}[h]
\centering
\includegraphics[width=0.55\textwidth]{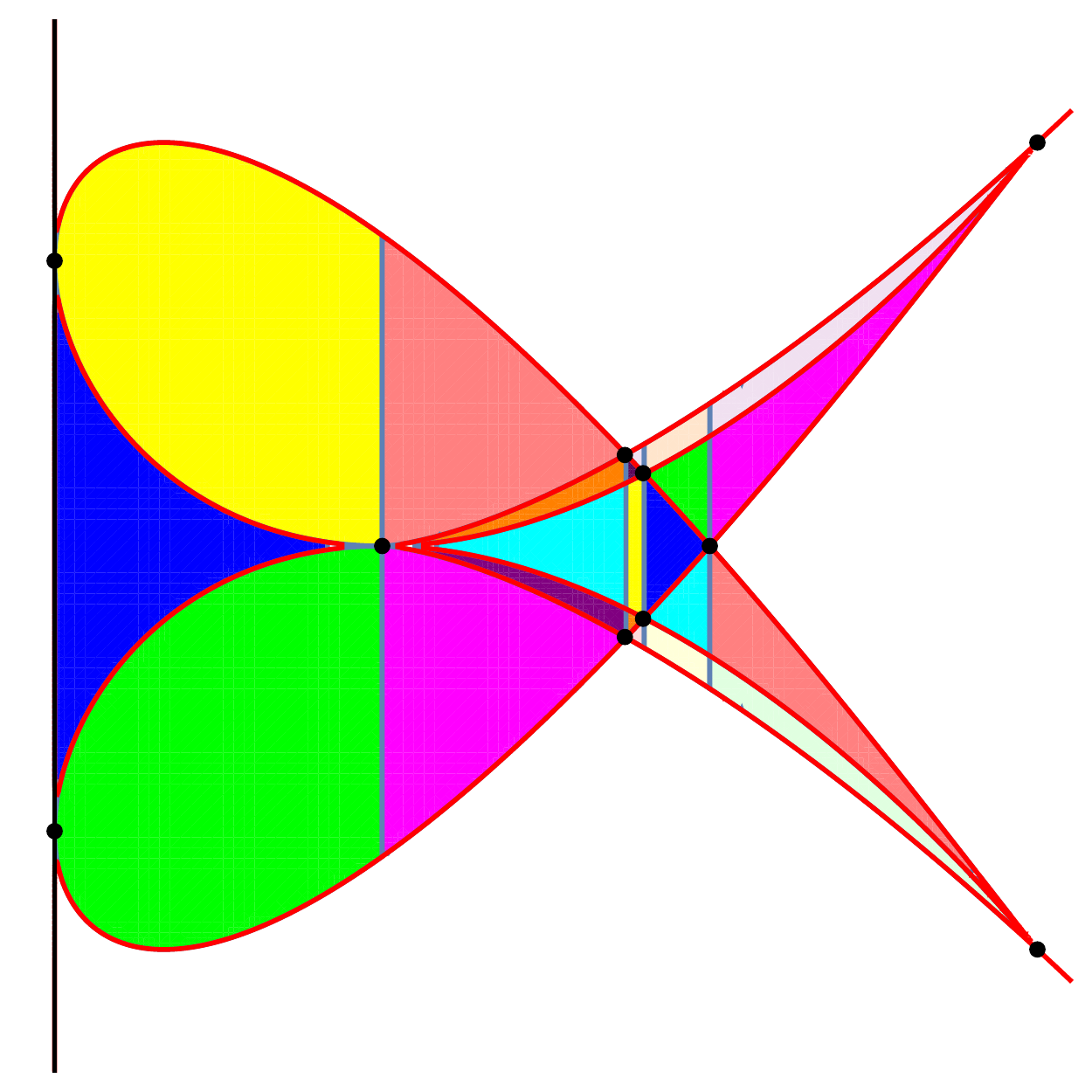}
\caption{
The flattened soccer ball in Examples~\ref{ex:fitness} and
\ref{example:fitness_example_with_our_algorithm} is divided into $22$ regions.
\label{figure:fitness_different_regions}}
\end{figure}

From our  output, we drew the picture 
in Figure~\ref{figure:fitness_different_regions}. For this, we  used
the coordinates of all singular points and branch points of $V(pq)$.
Branching occurs along the vertical tangent line $y = -\frac{1}{2}$.
The real singular points are $(-\frac{1}{2},-\frac{1}{3\sqrt{6}})$, $(-\frac{1}{2},\frac{1}{3\sqrt{6}})$,
$(0,0)$, $(\frac{16}{43},-\frac{64}{129\sqrt{129}})$, $(\frac{16}{43},\frac{64}{129\sqrt{129}})$, $(\frac{2}{5}$, $-\frac{2}{15\sqrt{15}})$, $(\frac{2}{5},\frac{2}{15\sqrt{15}})$, $(\frac{1}{2},0)$, $(1,-\frac{1}{3\sqrt{3}})$, $(1,\frac{1}{3\sqrt{3}})$.  Four pairs of critical points have the same $x$-coordinates, seen
in (\ref{eq:criticalvalues}).  The singular 
points and the vertical tangent line $y = -\frac{1}{2}$ are the black landmarks in  Figure~\ref{figure:fitness_different_regions}.
The $22$ regions are shown in different colors.  The six regions in $\phi(\mathcal{B}) \backslash \phi(\partial \mathcal{B})$ are colored light.
\end{example}

The output of Algorithm~\ref{algorithm:semialgebraic_sampling} can be interpreted as 
a semialgebraic formula for $\phi(\mathcal{B})$, as follows:
 Consider $p\cdot q$ as a univariate polynomial in $y$. Write down its Sturm sequence.
  Let $n$ denote the number of sign changes in that sequence evaluated at $-\infty$. The 
 formula for the region $(k,l)$ is a disjunction over all the possible sign assignments of the Sturm sequence with 
 $n-l$ sign changes, in conjunction with $x$ being between the $k$-th and $(k+1)$-st critical $x$-value. 
Such a formula will again be hard to read, just like the {\tt Mathematica} output
displayed at the end of Section \ref{sec:four}. A description like
Example \ref{example:fitness_example_with_our_algorithm},
accompanied  by a picture like Figure \ref{figure:fitness_different_regions}, seems to be
the most human-friendly way to represent  
the result of flattening a soccer ball.

We next discuss a more serious example,
where the {\tt Resolve} command does not terminate.
It will demonstrate the pros and cons of the exact symbolic approach.

\begin{figure}[h]
\centering
\includegraphics[width=0.5\textwidth]{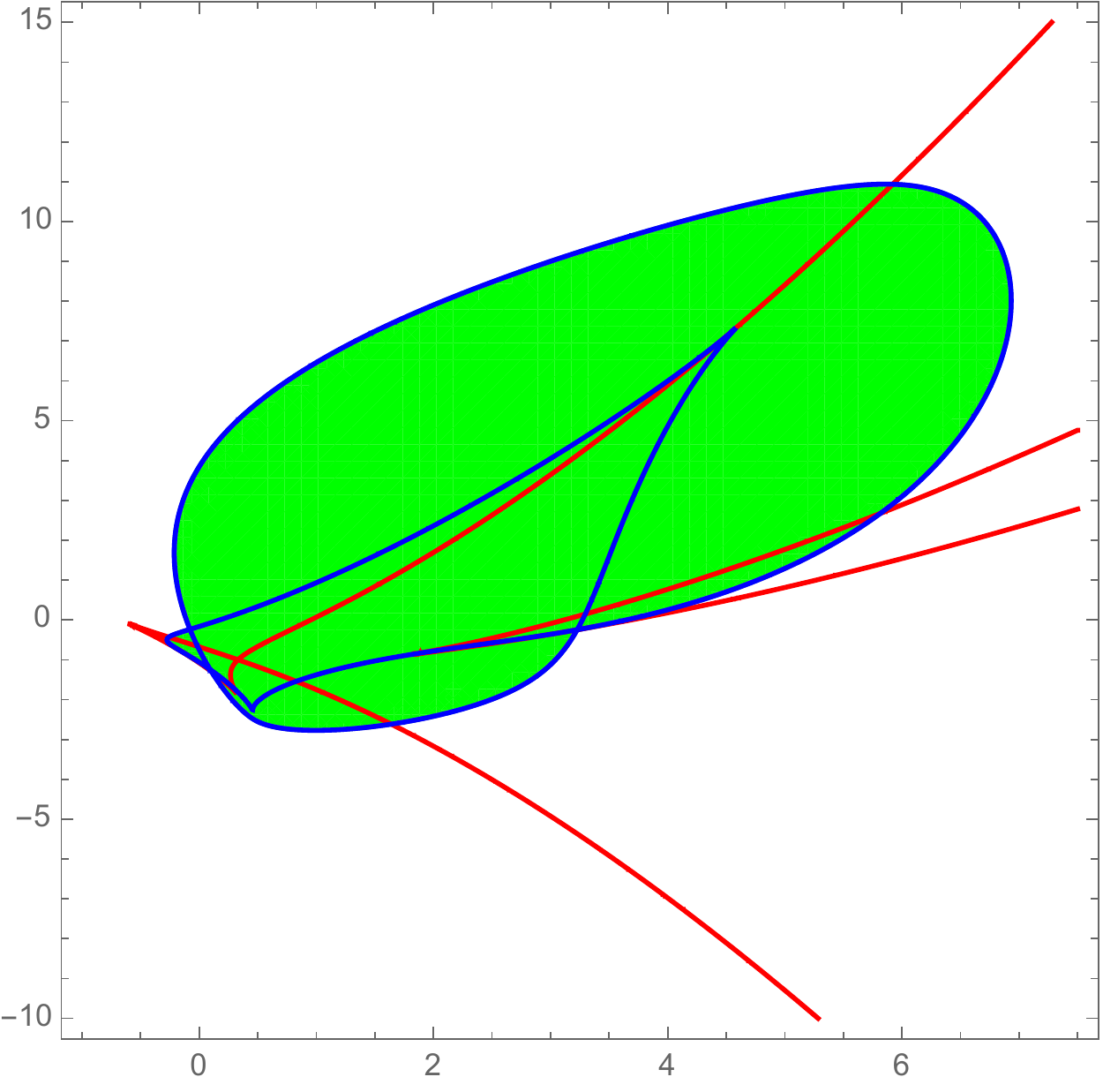}
\caption{The flattened soccer ball $\phi(\mathcal{B})$ in
Example~\ref{example:generic_example_algorithm} and its algebraic boundary.\label{figure:genericExample}}
\end{figure}

\begin{example}\label{example:generic_example_algorithm}
Fix $h = u^2+v^2+w^2-1$ as before, so $\mathcal{B}$ is the unit ball.
We select $f$ and $g$ randomly from
polynomials of degree $2$ and $3$ respectively.
This means we are now in the regime covered by
Theorem \ref{theorem:generic_degree_soccer}. Note 
the fourth line in Table \ref{tab:curves} marked $(d_1,d_2,e) = (2,3,2)$.
The following instance, picked for us by {\tt Macaulay2},
corresponds to the picture in Figure~\ref{figure:genericExample}:
$$
\begin{matrix}
f\,\,=\,\,\,&\frac{3}{5}u^2+uv+\frac{10}{3}v^2+\frac{7}{3}uw+\frac{1}{4}vw+\frac{3}{10}w^2+\frac{7}{4}u+\frac{8}{5}v+\frac{7}{5}w+\frac{10}{9},\\
g\,\,=\,\,\,&\frac{1}{4}u^3+3u^2v+uv^2+\frac{5}{3}v^3+u^2w+\frac{8}{5}uvw+\frac{4}{7}v^2w+\frac{7}{3}uw^2+\frac{7}{3}vw^2+\frac{7}{10}
      w^3+\\
&\frac{7}{2}u^2+3uv+\frac{5}{9}v^2+\frac{3}{8}uw+\frac{1}{9}vw+\frac{7}{4}w^2+\frac{9}{2}u+\frac{3}{4}v+\frac{5}{6}w+\frac{3}{7}.
\end{matrix}
$$

The polynomial $p$ that describes the branch locus of $\phi $
has degree $21$ and  $169$ terms. The polynomial $q$ that describes
the branch locus of $\phi_{\{h=0\}}$ has degree $24$ and $217$ terms. 
In both cases, these numbers count exactly the lattice points of the 
Newton triangles in Theorem~\ref{theorem:generic_degree_soccer}.
 
 For Step 3 we compute three univariate polynomials in $x$,  namely the
   resultants    in (\ref{eq:mustcompute}). 
 The resultant of $p$ and $\frac{\partial p}{\partial y}$  has degree $273$ with $13$ real roots. 
 The resultant of $p$ and $\frac{\partial q}{\partial y}$ has  degree $360$ with $22$ real roots.  
 The resultant of $p$ and $q$ has degree $336$ with $16$ real roots. 
 All $ 13+22+16$ real roots are different, so the total number of critical $x$-values is $51$.
 In Step 4 we sample points 
 from the cube $[-1,1]^3$ and from its boundary,
 and we record their images under $\phi$.
 In Step 5, we run over these points in $\R^2$, and we
identify the labels $(k,l)$ of the regions that contain these points.
Some regions are very small. It takes a long time to identify them.
We found that $\phi(\mathcal{B})$ is the union of the 
following $144$ regions:
\begin{small}
\begin{align*}
& (12, 2), (13, 2), (13, 3), (14, 2), (14,3), (14, 5), (15, 1), (15, 2), (15, 3), (15,5), (16, 1), (16, 2),\\
& (16, 3), (16, 4), (16,5), (17, 1), (17, 2), (17, 3), (17, 4), (17,5), (18, 1), (18, 2), (18, 3), (18, 4), \\
& (18,5), (19, 1), (19, 2), (19, 3), (19, 4), (19,5), (20, 1), (20, 2), (20, 3), (20, 4), (20,5), (20, 6),\\
& (20, 7), (21, 1), (21, 2), (21,3), (21, 4), (21, 5), (21, 6), (21, 7), (22,1), (22, 2), (22, 3), (22, 4),\\
& (22, 5), (23, 1), (23, 2), (23, 3), (23,4), (23, 5), (24, 1), (24, 2), (24, 3), (24,4), (24, 5), (24, 6), \\
& (24, 7), (25, 1), (25,2), (25, 3), (25, 4), (25, 5), (26,1), (26, 2), (26, 3), (26, 4), (26, 5), (27, 1),\\
& (27, 2), (27, 3), (27,4), (27, 5), (28, 2), (28, 3), (28, 4), (28,5), (29, 2), (29, 3), (29,4), (29, 5),\\
& (29, 6), (29, 7), (30, 2), (30,3), (30, 4), (30, 5), (30, 6), (30, 7), (31,2), (31, 3), (31, 4), (31, 5),\\
& (31, 6), (31,7), (32, 2), (32, 3), (32, 4), (32, 5), (32, 6), (32,7), (33, 3), (33, 4), (33, 5), (33, 6),\\
& (33, 7), (34,3), (34, 4), (34, 5), (34, 6), (34, 7), (35,3), (35, 4), (35, 5), (35, 6), (35, 7), (36,3), \\
& (36, 4), (36, 5), (36, 6), (36, 7), (37,3), (37, 4), (37, 5), (37, 6), (37, 7), (38,3), (38, 4), (38, 5),\\
& (38, 6), (38, 7), (39,3), (39, 4), (39, 5), (40, 3), (40,4), (40, 5), (41, 4), (41, 5), (42, 4), (43, 4).
\end{align*}
\end{small}
Out of the $50$ bounded segments between the critical $x$-values, 
precisely $32$ arise from $\mathcal{B}$.
The $11$ left-most segments and the $7$ right-most segments do not arise.
In the list above, the first pair $(12,2)$ refers to $x$-coordinates between critical $x$-values
labeled \#12 and \#13,
 which are approximately $-0.275436$ and $-0.2599$.
 The corresponding $y$-coordinates are between the
   $2$-nd and $3$-rd root of $p \cdot q$,
    regarded as a polynomial in $y$, with $x$ fixed in segment \#12.

\begin{figure}[h]
\centering
\begin{subfigure}[b]{0.45\textwidth}
\centering
\includegraphics[width=0.95\linewidth]{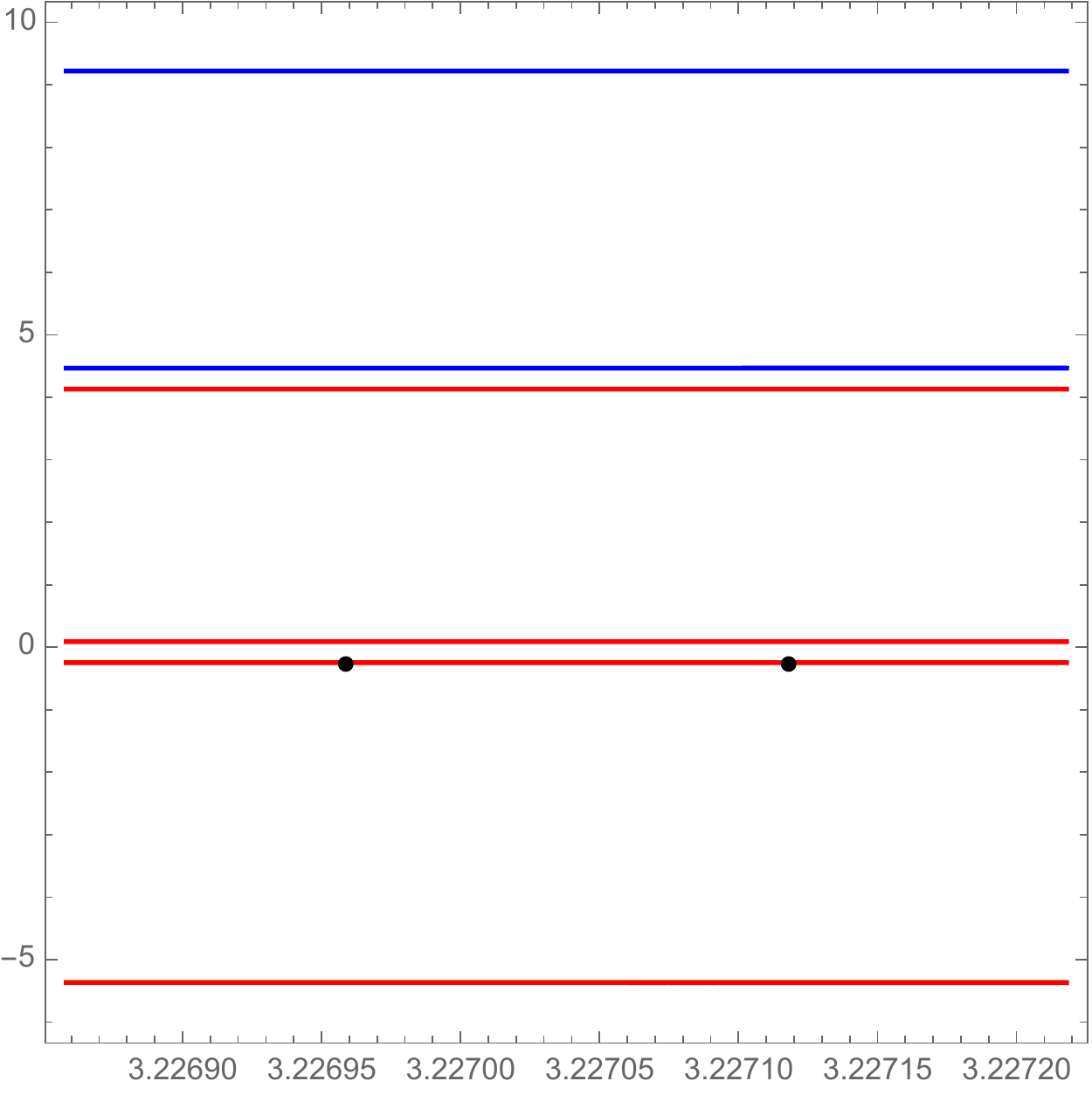}
\caption{Magnification of the $x$-axis only.}\label{figure:magnification1}
\end{subfigure} \qquad
\begin{subfigure}[b]{0.45\textwidth}
\centering
\includegraphics[width=\linewidth]{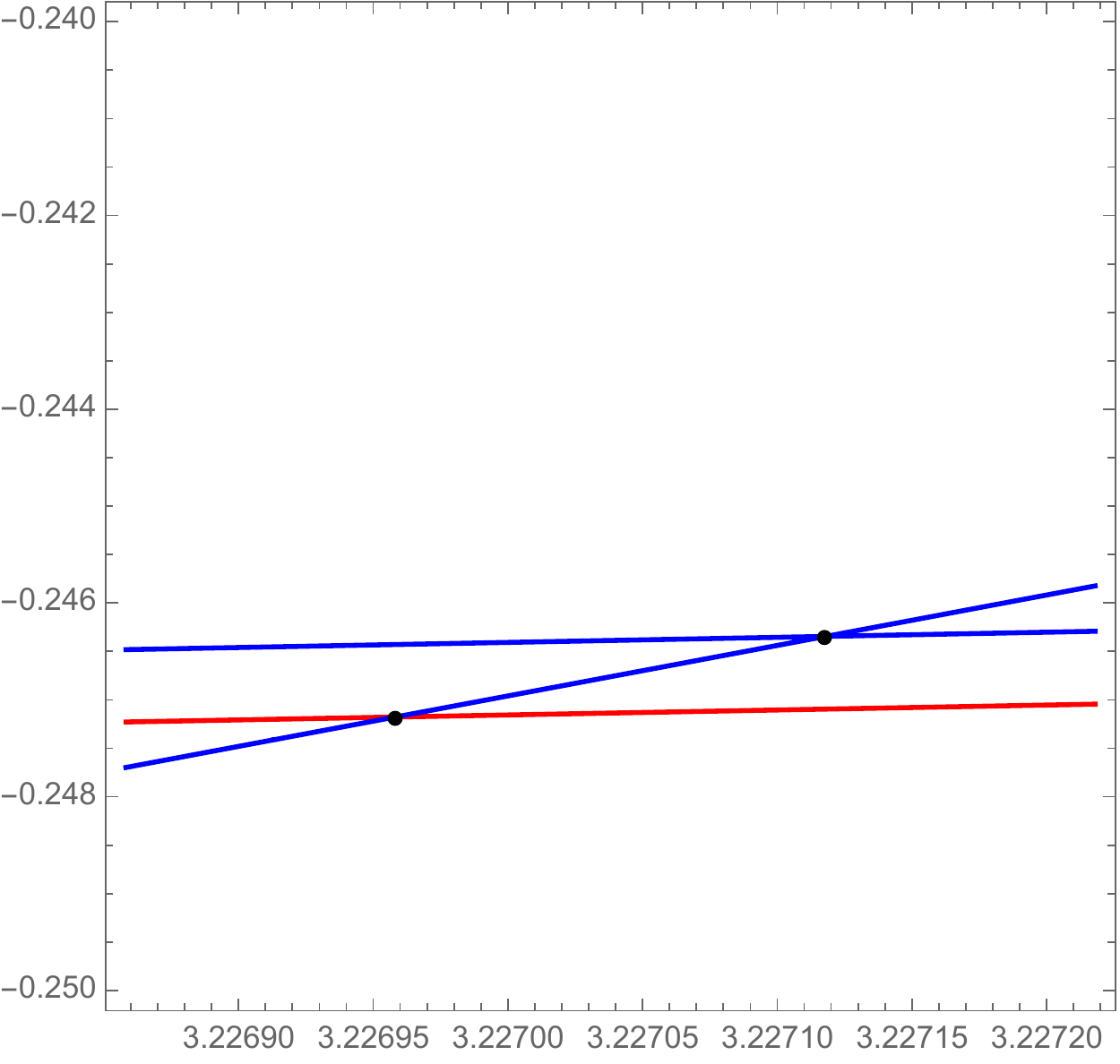}
\caption{Magnification of both coordinate axes.}\label{figure:magnification2}
\end{subfigure}
\caption{Magnifications between the 33rd and 34th critical points.}\label{figure:magnificationGenericExample}
\end{figure}

 Figure~\ref{figure:magnificationGenericExample} shows
 that the     situation is delicate. The regions can be
 very small, even for curves of degree $21$. For example, consider the critical $x$-values 
 labeled \#33 and \#34. They are
 $3.22696$ and $3.22712$. The unique critical point
over \#33 is an intersection point of $V(p)$ and $V(q)$.
 The unique critical point over \#34 is a node of $V(q)$.
 They are shown in Figure~\ref{figure:magnificationGenericExample}.

 In Figure~\ref{figure:magnification1} we scaled the $x$-axis so that it shows
 the vertical strip \#33. The $y$-axis is left unscaled, so that we can see
 all curve segments in that strip. It looks like the blue curve
 $V(p)$ has two segments and the red curve $V(q)$ has four segments.
 However,  two more blue segments 
 are eclipsed by the relevant red segment.
The truth becomes visible in Figure~\ref{figure:magnification2},
 where we also scale the $y$-axis. The critical point on the right
 is a node of the blue curve $V(q)$ and it does not lie on the red curve $V(p)$.
 The left critical point lies in $V(p,q)$.
 Our careful analysis also shows that $\phi(\mathcal{B}) \neq \phi(\partial \mathcal{B})$,  although this is not visible in Figure~\ref{figure:genericExample}.
Between critical $x$-values labeled \#15 and \#19, the lower
boundary is given by the red curve $V(p)$.
 \end{example}
  
\section{Convexity and Optimization}
\label{sec:five}

The problem of computing images of maps
is of considerable interest in polynomial optimization.
Magron, Henrion and Lasserre \cite{MHL}
developed a method for this based on outer approximations.
Our study is complementary to theirs,
in the sense that we do not consider approximations
but we seek exact descriptions. A related question is how to
compute and represent the convex hull of the image $\phi(\mathcal{B})$.
This issue will be addressed later in this section.

We start our discussion with a few examples
where both $\mathcal{B}$ and $\phi(\mathcal{B})$ are convex.
In what follows we retain the assumption that
$h = u^2+v^2+w^2-1$, so $\mathcal{B}$ is the unit ball in $\R^3$.
This can be folded into a convex polygon  in various interesting ways.

\begin{example} 
It is easy to find quadratic polynomials $f$ and $g$ such that
 $\phi = (f,g)$ maps the soccer ball onto a triangle or a rectangle.
Figure~\ref{fig:truesoccer} shows a map onto a square.
Here are two explicit maps that work.
If $ f = u^2+w^2$ and $g= v^2+w^2$
then $\phi(\mathcal{B})$ is the square with vertices
$(0,0),\,(0,1), \,(1,0), \,(1,1)$, and
$\phi(\partial \mathcal{B})$ is the triangle with vertices
$(0,1), \,(1,0), \,(1,1)$. The boundary polynomials are
$p = xy(x-y)$ and $q = (1-x)(1-y)(x+y-1)$.
For a second example let  $\phi$ be defined by
$ f = u^2+2v^2+w^2-1$ and $g = u^2+v^2+2w^2-1$.
Now the flattened soccer ball $\phi(\mathcal{B})$ is the triangle 
with vertices $(-1,-1),\,(0,1), \,(1,0)$,  while
$\phi(\partial \mathcal{B})$ is the triangle with vertices
$(0,0), \,(1,0), \,(0,1)$. The difference $\phi(\mathcal{B})\backslash \phi(\partial \mathcal{B})$
 is not a convex set.
\end{example}

If we pass from quadratic maps to cubic maps then we can create other polygons.

\begin{example} \label{ex:omega} 
Use the cubic Chebyshev polynomial $T_3(t) = 4t^3-3t$ to define $\phi$ via
\[
f(u,v,w) \,=\, \sqrt{3}(T_3(u)-T_3(v)) \qquad \hbox{and} \qquad
g(u,v,w) \,= \, T_3(u)+T_3(v)-2T_3(w).\]
The image $\phi(\mathcal{B})$ is the regular hexagon
 with vertices at  $(0,\pm 4)$ and $(\pm 2\sqrt{3},\pm 2)$.
\end{example}

This raises the question whether we can prescribe
$\phi(\mathcal{B})$ to be any polyhedral shape.
Ueno \cite{ueno1} proves that every unbounded convex polygon in
$\R^2$ is the image of $\R^2$ under a polynomial map.
It is not known, if his construction extends to
polynomial images of the unit ball.

\begin{problem}
Let $P$ be an arbitrary convex polygon in $\R^2$. Construct explicit
polynomials $f$ and $g$ in $\R[u,v,w]$ such that $P = \phi(\mathcal{B})$.
\end{problem}

Our next topic is the {\em flattenings of pancakes}.
These arise as special scenarios when we
flatten soccer balls. Indeed, suppose that the map
$\phi$ depends only on two of the variables,~say
$$ \phi \,: \,(u,v,w) \,\mapsto \,\bigl( \,f(u,v) ,\, g(u,v) \,\bigr). $$
Then the image of $\mathcal{B}$ is the same
as the image under $\phi'  = (f,g)$ of the unit disk 
$\,\mathcal{D}=\{(u,v) \in \mathbb{R}^2: u^2+v^2 \leq 1\}$.
In symbols, $\phi(\mathcal{B}) = \phi'(\mathcal{D})$.
The unit disk  $\mathcal{D}$ serves the role of our {\em pancake}.

 \begin{figure}[h]
\centering
\includegraphics[width=0.5\textwidth]{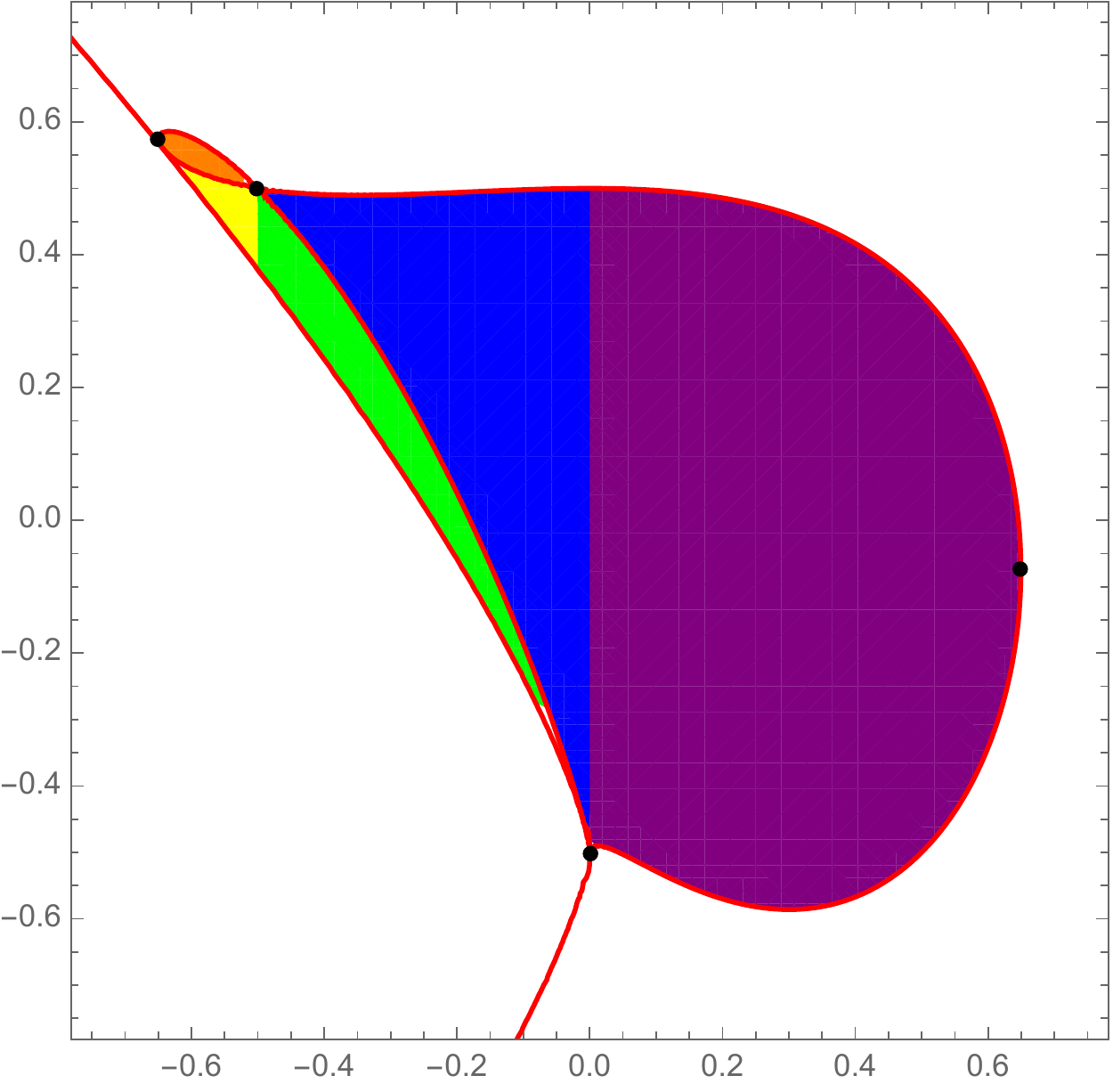}
\caption{\label{figure:pancake_with_critical_points}
This flattened pancake was discussed
 by Magron, Henrion and Lasserre~\cite{MHL}.}
\end{figure}

\begin{example}\label{example:pancake} 
Magron {\it et al.}~\cite[Example 1]{MHL} 
illustrate their method  for the map 
$$ f\,=\, \frac{1}{2}(u+uv) \quad \hbox{and}  \quad g\,=\, \frac{1}{2}(v-u^3). $$
It is instructive to compare their output with that of Algorithm \ref{algorithm:semialgebraic_sampling}.
The exact description  of the  flattened pancake
begins with the
two polynomials that define the algebraic boundary:
$$
\begin{matrix}
p & = & 2048x^3+432y^4+864y^3+648y^2+216y+27, \qquad \\
q & = & 64x^6+128x^5+96x^4+128x^3y-32x^3+192x^2y^2 - 44 x^2\\
 & & + 96xy^3+48xy^2-24xy-12x+16y^4+16y^3-4y-1 .\\
 \end{matrix}
 $$
The projection of the curve $V(pq)$ onto the $x$-axis has four critical points:
$$
\bigl(-\frac{3 \sqrt{3}}{8}, \frac{-473 + 264 \sqrt{3}}{6208 - 3600 \sqrt{3}} \bigr),\,\,
 \bigl(-\frac{1}{2}, \frac{1}{2}\bigr), \,\, \bigl(0, -\frac{1}{2} \bigr),\,\,
  \bigl(\frac{3 \sqrt{3}}{8}, \frac{-473 - 264 \sqrt{3}}{6208 + 3600 \sqrt{3}} \bigr).
$$
Our quantifier-free description of the planar region $\phi'(\cal{D})$ 
consists of the five pairs:
$$ (1,2)\,, \,\, (1,3) \,,\,\, (2,2) \,,\,\, (2,3) \,,\,\, (3,1) .$$
Figure~\ref{figure:pancake_with_critical_points} shows 
 $\phi'(\mathcal{D})$ where these
five regions are colored in yellow, orange, green, blue and purple.
The critical points are black, and the  boundary curve
$V(pq)$ is red.
 \end{example}

Outer approximations of the kind studied in \cite{MHL}
tend to work best when one is interested not in the image itself
but  in its convex hull  ${\rm conv}(\phi(\mathcal{B}))$.
Its extreme points are the solutions for
 the parametrized family of optimization problems
\begin{equation}
\gamma(\alpha,\beta) \,\,:= \,\,\max \{ 
\alpha x + \beta y \, : \,  (x,y) \in \phi(\mathcal{B}) \}.
\label{eq:mingamma}
\end{equation}
The function $\gamma$ is called the \emph{support function} of the set $\phi(\mathcal
{B})$. From the values of this function one obtains a description of the convex hull
of $\phi(\mathcal{B})$ as an intersection of closed halfspaces:
\[
\mathrm{conv}\bigl( \phi(\mathcal{B}) \bigr) \,\,= \bigcap_{(\alpha,\beta) \in \R^2} \, \{ (x,y) : \alpha x + \beta y \leq \gamma(\alpha,\beta) \}.
\]

The optimization problem in \eqref{eq:mingamma} can be
equivalently written as
\[
\gamma(\alpha,\beta) \,\,= \,\max_{(u,v,w)}  \, 
\alpha f(u,v,w) + \beta g(u,v,w) \qquad \mbox{s.t.} \quad h(u,v,w) \geq 0.
\]
Both the objective function and the constraint are given by polynomial
functions. Although in general this can be difficult to solve, good
upper bounds can be obtained using \emph{sum of squares} methods
\cite{BPT}. This method minimizes $\gamma_{SOS}$ subject~to
\[
\gamma_{SOS} - (\alpha f(u,v,w) + \beta g(u,v,w)) \,\,= \,\,s_0(u,v,w) + s_1(u,v,w) h(u,v,w),
\]
where $s_0$ and $s_1$ are \emph{sums of squares}. The solution satisfies
 the inequality $\gamma \leq \gamma_{SOS}$. Furthermore,
restricting the polynomials $s_0,s_1$ to have fixed degree, this is
a semidefinite optimization problem, so it can be solved efficiently.
Since $\gamma(\alpha,\beta) \leq \gamma_{SOS}(\alpha,\beta)$,
we obtain an outer approximation of the convex hull:
\begin{equation}
\label{eq:outerappr}
\mathrm{conv} \bigl( \phi(\mathcal{B})\bigr) \,\,\, \subseteq \,
\bigcap_{(\alpha,\beta) \in   \R^2} \, \bigl\{ (x,y) : \alpha x + \beta y \leq
\gamma_{SOS}(\alpha,\beta) \bigr\}.
\end{equation}
The right hand side is a {\it spectrahedral shadow}, so it is
a desirable set in the context of \cite{BPT}.
If one is lucky then equality holds in (\ref{eq:outerappr}) 
and a semidefinite representation of $\, \mathrm{conv} ( \phi(\mathcal{B}))\, $
has been obtained.
From our earlier algebraic perspective, such a representation still involves
quantifiers. Any quantifier-free formula has to account for
 supporting lines that were created when 
passing from  $\phi(\mathcal{B})$ to its convex hull.
Those lines are  bitangents of our curve $V(pq)$.

\acknowledgement{Pablo Parrilo was supported by AFOSR FA9550-11-1-0305.
  Bernd Sturmfels was supported by NSF grant  DMS-1419018
  and the Einstein Foundation Berlin. 
Part of this work was done while the authors visited the Simons
Institute for the Theory of Computing at UC Berkeley.}

%
%

\begin{thebibliography}{99.}%
\bibitem{BR2} S.~Basu and C.~Riener:
{\em Bounding the equivariant Betti numbers of symmetric semi-algebraic sets},
Advances in Mathematics~{\bf 305} (2017), 803--855.

\bibitem{BR} S.~Basu and A.~Rizzie:
{\em Multi-degree bounds on the Betti numbers of real varieties and semi-algebraic sets and applications},
{\tt arXiv:1507.03958}.

\bibitem{BPT} G.~Blekherman, P.A.~Parrilo and R.~Thomas:
{\em Semidefinite optimization and convex algebraic geometry},
MOS-SIAM Series on Optimization~{\bf 13}, SIAM, 2013.

\bibitem{BHJS}
M.~Bogle, J.~Hearst, V.~Jones and L.~Stoilov:
{\em  Lissajous knots},
J. Knot Theory Ramifications~{\bf 3} (1994), 121--140.

\bibitem{BDEMW}
R.~Bradford, J.~Davenport, M.~England, S.~McCallum and D.~Wilson:
{\em Truth table invariant cylindrical algebraic decomposition}, J.~Symbolic Comput.~{\bf 76} (2016), 1--35. 

\bibitem{Brickman}
L.~Brickman:
{\em On the field of values of a matrix}, 
Proceedings of the American Mathematical Society~{\bf 12} (1961), 61--66.

\bibitem{Brown}
C.W.~Brown: 
{\em QEPCAD B: A program for computing with semi-algebraic sets using CADs}, SIGSAM
Bull.~{\bf 37} (2003), 97--108.

\bibitem{Collins}
G.E.~Collins:
{\em Quantifier elimination for real closed fields by cylindrical algebraic decompostion},
Springer Lecture Notes in Computer Science~{\bf 33} (1975), 134--183.

\bibitem{EWBD}
M.~England, D.~Wilson, R.~Bradford and J.H. Davenport:
{\em Using the Regular Chains Library to build cylindrical algebraic decompositions by projecting and lifting},
in Mathematical Software -- ICMS 2014, Springer, 2014.

\bibitem{M2} D.~Grayson and M.~Stillman:
{\em Macaulay2, a software system for research in algebraic geometry}, available at
{\tt www.math.uiuc.edu/Macaulay2/}.

\bibitem{HE}
H.~Hong and M.~Safey El Din:
{\em Variant quantifier elimination},
J.~Symbolic Comput.~{\bf 47} (2012), 883--901.

\bibitem{Joh}
T.~Johnsen: {\em Plane projections of a smooth space curve},
 in Parameter Spaces, 89--110, Banach Center Publications~{\bf 36}, Polish Academy of Sciences, 1996.

\bibitem{KKKR}
T.~Kahle, K.~Kubjas, M.~Kummer and Z.~Rosen:
{\em The geometry of rank-one tensor completion},
SIAM J.~Appl.~Algebra Geometry (2017).

\bibitem{LMX}
F.~Lemaire, M.~Moreno Maza and Y.~Xie: 
{\em The RegularChains library in MAPLE}, SIGSAM Bull.~{\bf 39} (2005), 96--97. 

\bibitem{MS}
D.~Maclagan and B.~Sturmfels: {\em Introduction to Tropical Geometry},
 Graduate Studies in Mathematics~{\bf 161}, American Mathematical Society, 2015. 

\bibitem{MHL}
V.~Magron, D.~Henrion and J.-B.~Lasserre:
{\em 	 Semidefinite approximations of projections and polynomial images of semialgebraic sets},
SIAM J.~Optim.~{\bf 25} (2015), 2143--2164.

\bibitem{Manivel} L.~Manivel:
{\em Symmetric Functions, Schubert Polynomials, and Degeneracy Loci},
 SMF/AMS Texts and Monographs~{\bf 6}, Cours Sp\'ecialis\'es~{\bf 3},
 American Mathematical Society, Soci\'et\'e Math\'ematique de France, 2001.
 
\bibitem{Renegar} J.~Renegar: 
{\em On the computational complexity and geometry of the first-order theory of the
reals}, J. Symb. Comput.~{\bf 13} (1992), 255--352.
 
 \bibitem{Sinn}
R.~Sinn: {\em Algebraic boundaries of SO(2)-orbitopes},
Discrete Comput.~Geometry~{\bf 50} (2013), 219--235.

\bibitem{SiSt} R.~Sinn and B.~Sturmfels:
{\em Generic spectrahedral shadows}, 
SIAM J.~Optim.~{\bf 25} (2015), 1209--1220.
 
\bibitem{ueno1}
C.~Ueno: {\em On convex polygons and their complements as images of regular and
 polynomial maps of $\R^2$}, J. Pure Appl. Algebra~{\bf 216} (2012), 2436--2448.
\end{thebibliography}
%

\end{document}